\documentclass[reqno]{amsart}
\usepackage{amssymb,latexsym}
\usepackage{amsmath}
\usepackage{amsthm}
\usepackage{graphicx}
\usepackage{hyperref}
\usepackage{titletoc}
\numberwithin{equation}{section}
\newtheorem{theorem}{Theorem}[section]
\newtheorem{proposition}[theorem]{Proposition}
\newtheorem{lemma}[theorem]{Lemma}
\newtheorem{corollary}[theorem]{Corollary}

\newtheorem{Theorem}{Theorem}[section]

\theoremstyle{definition}

\def \O{\Omega_{\varepsilon}}

\def\pt{\partial}
\def\B{\Bbb R^{n}}
\def\e{\varepsilon}
\def\cal{\mathcal}
\def\l{\langle}
\def\r{\rangle}

\begin{document}
\title{On Lin-Ni's conjecture in dimensions four and six}

\author{  Juncheng Wei }
\address{ Juncheng ~Wei,~Department of Mathematics, University of British Columbia, Vancouver, BC V6T 1Z2 }
\email{jcwei@math.ubc.ca}
\author{  Bin Xu}
\address{ Bin ~Xu,~School of Mathematics and  Statistics, Jiangsu ~Normal University, Xuzhou, Jiangsu, 221116,~~~~ People Republic of China}
\email{dream-010@163.com}
\author{  Wen Yang }
\address{ Wen ~Yang,~Department of Mathematics,
Chinese University of Hong Kong, Shatin, Hong Kong  }
\email{math.yangwen@gmail.com}

\begin{abstract}
We give negative answers to Lin-Ni's conjecture for any four and six dimensional domains. No condition on the symmetry, geometry nor topology of the domain is needed.
\end{abstract}

\date{}\maketitle

\section{Introduction}\par
This paper is concerned  with the following nonlinear elliptic Neumann  problem:
\begin{equation}
\label{1.1}
\left\{\begin{array}{ll}
\Delta u-\mu u+u^{q}=0,~~~u>0~~& {\rm in}~\Omega,\\
\frac {\partial u}{\partial \nu}=0,~&{\rm on}~\partial \Omega,
\end{array}\right.
\end{equation}
where $1<q<+\infty,~\mu>0$~and ~$\Omega$~is a smooth and bounded domain in $\mathbb{R}^n$ ($n\geq 2$).

\medskip

Equation (\ref{1.1}) arises in many branches of the applied sciences. For example, it can be viewed as a steady-state equation for the shadow system of the Gierer-Meinhardt system in mathematical biology \cite{gm}, \cite{ni},
or for parabolic equation in chemotaxis, e.g. Keller-Segel model \cite{lnt}. When $q=\frac{n+2}{n-2}$, it can also be viewed as a Brezis-Nirenberg type Neumann problem \cite{BN}.

Equation (\ref{1.1}) has  at least one solution, namely the constant solution $u\equiv \mu^{\frac {1}{q-1}}.$  It turns out that this is the only solution, provided that $\mu $ is small and $q<\frac{n+2}{n-2}$. This was first proved by Lin-Ni-Takagi \cite{lnt}, via blow up analysis and compactness argument. Based on this, Lin and Ni \cite{ln}  made the following conjecture:

\medskip

\noindent
{\em {\bf Lin-Ni's Conjecture \cite{ln}:} For $\mu$  small and $q = \frac{n+2}{n-2}$,  problem (\ref{1.1}) admits only the constant solution.}

\medskip

In recent years, many progress have been made towards the understanding of Lin-Ni's conjecture.

The first result was due to Adimurthi-Yadava \cite{ay1}-\cite{ay2} (and independently  Budd-Knapp-Peletier \cite{bkp}). They  considered radial solutions of the
following problem
\begin{equation}
\label{1.2}
\left\{\begin{array}{ll}
\Delta u-\mu u+u^{\frac{n+2}{n-2}}=0\quad&\mathrm{in}~B_R(0),\\
u= u(|x|),\ \ u>0~&\mathrm{in}~B_R(0),\\
\frac{\partial u}{\partial \nu}=0,\quad&\mathrm{on}~\partial B_R(0)
\end{array}\right.
\end{equation}
and the following results were proved

\begin{Theorem}
\label{th1.1}
\rm{(}\cite{ay1}-\cite{ay3},\cite{bkp}\rm{)}~For $\mu$~sufficiently small\par
{\rm (1)}~if $n=3$ or $n \geq 7$,~problem (\ref{1.2}) admits only the constant solution;\par
{\rm (2)}~if~ $n=4,5, 6$,~problem (\ref{1.2}) admits a nonconstant solution.
\end{Theorem}

Theorem A reviews the dimension effects on Lin-Ni's conjecture. However the proof of Theorem \ref{th1.1} depends on the radial symmetry of the domain and the solution and thus is difficult to generalize to general domains.  In the general domain case, the complete answer is not known yet, but there are a few results. In the general three dimensional domain case, Zhu \cite{z} and Wei-Xu \cite{wx} proved

\begin{Theorem}
\label{th1.2}
\rm{(}\cite{wx},\cite{z}\rm{)}~The conjecture is true if $n=3~(q=5)$~and $\Omega$~is convex.
\end{Theorem}

Zhu's proof relies on blowing up analysis and a priori estimates,  while Wei-Xu \cite{wx} gave a direct proof of Theorem \ref{th1.2} by using only integration by parts.

\medskip

Part (1) of Theorem A is generalized  by Druet-Robert-Wei \cite{drw} to mean convex domains with bounded energy.

\begin{Theorem}
\label{th1.3}
\rm{(}\cite{drw}\rm{)}~Let $\Omega$ be a smooth bounded domain of $\mathbb{R}^n$, $n=3$ or $n\geq7.$ Assume that $H(x)>0$ for all $x\in\partial\Omega,$ where $H(x)$ is the mean curvature of $\partial\Omega$ at $x\in\partial\Omega.$ Then for all $\mu>0,$ there exists $\mu_0(\Omega,\Lambda)>0$ such that for all $\mu\in(0,\mu_0(\Omega,\Lambda))$ and for any $u\in C^2(\overline{\Omega}),$ we have that
\begin{equation*}
\left\{\begin{array}{ll}
\Delta u+\mu u=u^{2^*-1}\quad&\mathrm{in}~\Omega\\
u>0&\mathrm{in}~\Omega\\
\frac{\partial u}{\partial \nu}=0&\mathrm{on}~\partial\Omega\\
\int_{\Omega}u^{2^*}dx\leq\Lambda&
\end{array}\right\}\Rightarrow u\equiv\mu^{\frac{n-2}{4}}.
\end{equation*}
\end{Theorem}
\vspace{0.3cm}

It should be mentioned that the assumption of  bounded energy in Theorem C is necessary. Without this technical assumption, it was proved that solutions to (\ref{1.1}) may accumulate with infinite energy when the mean curvature is negative somewhere (see Wang-Wei-Yan \cite{wwy}). More precisely, Wang-Wei-Yan gave a negative answer to Lin-Ni's conjecture in all dimensions ($n\geq3$) for non-convex domain by assuming that $\Omega$ is a smooth and bounded domain satisfying the following conditions:\\

\noindent ($H_1$)~$y\in\Omega$ if and only if $(y_1,y_2,y_3,\cdots,-y_i,\cdots,y_n)\in\Omega,~\forall i=3,\cdots,n.$ \par
\noindent($H_2$)~If $(r,0,y'')\in\Omega$, then $(r\cos\theta,r\sin\theta,y'')\in\Omega,~\forall\theta\in(0,2\pi),$ where $y''=(y_3,\cdots,y_n).$\par
\noindent($H_3$)~Let $T:=\partial\Omega\cap\{y_3=\cdots=y_n=0\}$. There exists a connected component $\Gamma$ of $T$ such that $H(x)\equiv\gamma<0,~\forall x\in\Gamma.$

\begin{Theorem}
\label{th1.4}
\rm{(}\cite{wwy}\rm{)}~Suppose $n\geq3,~q=\frac{n+2}{n-2}$ and $\Omega$ is a bounded smooth domain satisfying ($H_1$)-($H_3$). Let $\mu$ be any fixed positive number. Then problem (\ref{1.1}) has infinitely many positive solutions, whose energy can be made arbitrarily large.
\end{Theorem}

Wang-Wei-Yan \cite{wwy2} also gave a negative answer to Lin-Ni's conjecture in some convex domain including the balls for $n\geq 4.$

\begin{Theorem}
\label{th1.5}\rm{(}\cite{wwy2}\rm{)}~Suppose $n\geq4,~q=\frac{n+2}{n-2}$ and $\Omega$ satisfies ($H_1$)-($H_2$). Let $\mu$ be a any fixed positive number. Then problem (\ref{1.1}) has infinitely many positive solutions, whose energy can be made arbitrarily large.
\end{Theorem}

Theorems \ref{th1.1}-\ref{th1.5} reveal that Lin-Ni's conjecture depends very sensitively not only on the dimensions, but also on the shape of the domain (convexity). A natural question is: what about the general domains?

 \medskip
So far  the only result for general domains is given by Rey-Wei \cite{rw} in which they disproved the conjecture in the five-dimensional case by constructing  an nontrivial solution which blows up at $K$ interior points in $\Omega$ provided $\mu$ is sufficiently small.  In view of  results of Theorem A, Rey and Wei \cite{rw} conjectured that  we should  have a negative answer to Lin-Ni's conjecture  in all the  dimensions $n=4,5,6.$  This is exactly what we shall achieve  in this paper.

\medskip

The purpose of this paper is to establish a result similar to (2) of Theorem A in general four, and six-dimensional domains by constructing  a nontrivial solution which blows up at a single point in $\Omega$ provided $\mu$ is sufficiently small. From now on, we consider the problem
\begin{equation}
\label{1.3}
\Delta u-\mu u+u^{\frac{n+2}{n-2}}=0~\mathrm{in}~\Omega,\quad u>0~\mathrm{in}~\Omega,\quad\frac{\partial u}{\partial \nu}=0~\mathrm{on}~\partial\Omega,
\end{equation}
where $n=4,6$ and $\Omega$ is a smooth bounded domain in $\mathbb{R}^{n}$ and $\mu>0$ very small. Our main result is stated as follows\\

\noindent {\bf Main Theorem.}
For problem (\ref{1.3}) in $n=4,6$, there exists $\mu_0>0$ such that for all $0<\mu<\mu_0,$ equation (\ref{1.3}) possesses a nontrivial solution which blows up at an interior point of
$\Omega.$\\

Combining with the results in \cite{rw}, we have the following corollary.

\begin{corollary}
When $n=4,5,6$, Lin-Ni's conjecture is false for general domains.
\end{corollary}

In order to make more precise statement of the Main Theorem, we introduce the following notation. Let $G(x,Q)$~be the Green's function defined as
\begin{equation}
\label{1.4}
\Delta_{x}G(x,Q)+\delta_{Q}-\frac {1}{|\Omega|}=0~~{\rm in}~\Omega,~~\frac {\pt G}{\pt \nu}=0~~{\rm on}~\pt \Omega,~~~\int_{\Omega}G(x,Q)dx=0.
\end{equation}

We decompose
$$G(x,Q)=K(|x-Q|)-H(x,Q),$$
where
\begin{equation}
\label{1.5}
K(r)=\frac {1}{c_nr^{n-2}},~c_n=(n-2)|S^{n-1}|,
\end{equation}
is the fundamental solution of the Laplacian operator in $\Bbb R^n(|S^{n-1}|$~denotes the area of the unit sphere),~$n=4,6$.\par

For the reason of normalization, we consider throughout the paper the following equation:
\begin{equation}
\label{1.6}
\Delta u-\mu u+n(n-2)u^{\frac {n+2}{n-2}}=0,~~u>0~~{\rm in}~\Omega,~~\frac {\pt u}{\pt \nu}=0~~{\rm on}~\pt \Omega.
\end{equation}\par

We recall that, according to \cite{cgs}, the functions
\begin{equation}
\label{1.7}
U_{\Lambda, Q}=(\frac {\Lambda}{\Lambda^2+|x-Q|^2})^{\frac {n-2}{2}},~\Lambda>0,~Q\in \Bbb R^n,
\end{equation}
are the only solutions to the problem
\begin{equation}
\label{1.8}
-\Delta u=n(n-2)u^{\frac {n+2}{n-2}},~~u>0~~~{\rm in}~\Bbb R^n.
\end{equation}

Our main result can be stated precisely as follows:
\begin{theorem}
\label{th1.6}
Let $\Omega$ be any smooth bounded domain in $\mathbb{R}^n.$ \par
{\rm(1).}~For $n=4,$ there exists $\mu_1>0$ such that for $0<\mu<\mu_1,$ problem (\ref{1.6}) has a nontrivial solution
$$u_{\mu}=U_{e^{-\frac{c_1}{\mu^2}}\Lambda,Q^{\mu}}+O(\mu^{-1}e^{-\frac{c_1}{\mu^2}}),$$
where $c_1$ is some constant depending on the domain, to be determined later, $\Lambda$ will be some generic constant. The blow up point $Q$ depends on the domain and parameter $\Lambda.$
\par
{\rm(2).}~For $n=6$, there exists $\mu_2>0$ such that for $0<\mu<\mu_2,$ problem (\ref{1.6}) has a nontrivial solution
$$u_{\mu}=U_{\mu\Lambda,Q^{\mu}}+O(\mu),$$
where $\Lambda\rightarrow\Lambda_0,$ and $\Lambda_0>0$ is some generic constant. The blow up point $Q$ depends on the domain and parameter $\Lambda.$
\end{theorem}

We introduce several notations for late use. Set
\begin{equation}
\label{1.9}
\Omega_{\e}:=\Omega/\e=\{z|\e z\in \Omega\},
\end{equation}
and
\begin{align}
\label{1.10}
\mu=\begin{cases}(\frac {c_1}{-\ln \e})^{\frac 12},~~& n=4,\\
~\e~,~~& n=6.\end{cases}
\end{align}

Through the transformation $u(x)\longmapsto \e^{-\frac {n-2}{2}}u(x/\e),$ (\ref{1.6}) becomes
\begin{equation}
\label{1.11}
\Delta u-\mu\e^2u+n(n-2)u^{\frac {n+2}{n-2}}=0,~~~~u>0~~{\rm in}~\Omega_{\e},~~\frac {\pt u}{\pt \nu}=0~~{\rm on}~ \Omega_{\e}.
\end{equation}

We set
\begin{equation}
\label{1.12}
S_{\e}[u]:=-\Delta u+\mu\e^2u-n(n-2)u_+^{\frac {n+2}{n-2}},~u_+=\max(u,0),
\end{equation}
and introduce the following functional
\begin{equation}
\label{1.13}
J_{\e}[u]:=\frac{1}{2}\int_{\O}|\nabla u|^2+\frac{1}{2}\mu\e^2\int_{\O}u^2-\frac {(n-2)^2}{2}\int_{\O}|u|^{\frac{2n}{n-2}},~~u\in H^1({\O}).
\end{equation}

\medskip

Depending on the dimensions, we have to overcome different difficulties. In dimension four, the main problem is that the relation between $\mu$ and $\epsilon$ is only implicit. Dimension six is the {\em borderline} case, since  in the linearized operator the constant term $-\mu u$ disappears. To remedy this problem, we have to introduce an artificial parameter $\eta$ (see (\ref{2.14})). This case can be considered as "resonance" case because the constant lies in the kernel of the outer problem.

\medskip

The paper is organized as follows: In Section 2, we construct suitable approximated bubble solution $W,$~and state their properties. In Section 3,~we solve the linearized problem at $W$~ up to a finite-dimensional space. Then, in Section 4, we are able to solve the nonlinear problem in that space. In section 5, we study the remaining finite-dimensional problem and solve it in Section 6, finding critical points of the reduced energy functional. Some numerical results may be found in the last Section.

\medskip

\noindent
{\bf Acknowledgements:} The research of Wei is supported by a NSERC of Canada. Part of the paper was finished when the second author was visiting Chinese University of Hong Kong. He would like to thank the institute for their warm hospitality.

\section{Approximate bubble solutions}
\setcounter{equation}{0}
In this section, we construct suitable approximate solution, in the neighborhood of which solutions in Theorem \ref{th1.6} will be found. Depending on the dimensions, we shall make different ansatz.


Let $\mu$ and $\e$ be as defined in (\ref{1.10}). For any $Q\in\Omega$ with $d(Q/\e,\partial\O)$ large, $U_{\Lambda,Q/\e}$ in (\ref{1.7}) provides an approximate solution of (\ref{1.11}). Because of the appearance of the additional linear term $\mu\varepsilon^2u$ and the homogeneous Neumann boundary condition, we need to add two extra terms to get a better approximation. Now we describe the next order terms in different dimensions.

When $n=4$, we consider the following  linear equation
 \begin{equation}
\label{2.1}
\Delta\bar{\Psi}+U_{1,0}=0~~~{\rm in }~\Bbb R^4,~~~\bar {\Psi}(0)=1.
\end{equation}
which has q unique radial solution with the following asymptotic behavior
\begin{equation}
\label{2.2}
\bar {\Psi}(|y|)=-\frac 12\ln|y|+I+O\Big{(}\frac {1}{|y|}\Big {)},~~\bar {\Psi}^{'}=-\frac {1}{2|y|}\Big{(}1+O\Big{(}\frac {\ln (1+|y|)}{|y|^2}\Big {)}\Big {)}~~{\rm as}~|y|\rightarrow \infty,
\end{equation}
where $I$~is a generic constant. For $ Q \in \Omega_\epsilon$, set
\begin{equation}
\label{2.3}
\Psi_{\Lambda,Q}=\frac {\Lambda}{2}\ln \frac {1}{\Lambda \e}+\Lambda \bar {\Psi}(\frac {y-Q}{\Lambda}).
\end{equation}
which satisfies
$$\Delta \Psi_{\Lambda,Q}+U_{\Lambda,Q}=0 \ \ \ \mbox{in} \ \mathbb{R}^4.$$
From (\ref{2.2}) we derive that
\begin{equation}
\label{2.4}
|\Psi_{\Lambda,Q}(y)|,~|\partial_{\Lambda}\Psi_{\Lambda,Q}(y)|\leq C\Big{|}\ln \frac {1}{\e(1+|y-Q|)}\Big{|},~~|\partial_{Q_i}\Psi_{\Lambda,Q}(y)|\leq \frac {C}{1+|y-Q|}.
\end{equation}

Now we turn to the case of $n=6$.  Let $\Psi(|y|)$ be the radial solution of
\begin{equation}
\label{2.5}
\Delta\Psi+U_{1,0}=0~\mathrm{in}~\mathbb{R}^6,\quad\Psi\rightarrow0~\mathrm{as}~|y|\rightarrow+\infty.
\end{equation}
Then it is easy to check that
\begin{equation}
\label{2.6}
\Psi(y)=\frac{1}{4|y|^2}(1+O(\frac{1}{|y|^2}))~\mathrm{as}~|y|\rightarrow+\infty.
\end{equation}
For $Q\in\O,$ we set
$$\Psi_{\Lambda,Q}(y)=\Psi(\frac{y-Q}{\Lambda}).$$
Then it satisfies
$$\Delta\Psi_{\Lambda,Q}(y)+U_{\Lambda,Q}=0~\rm{in}~\mathbb{R}^6.$$
It is easy to check that
\begin{equation}
\label{2.7}
|\Psi_{\Lambda,Q}(y)|,~|\partial_{\Lambda}\Psi_{\Lambda,Q}(y)|\leq\frac{C}{(1+|y-Q|)^2},
~|\partial_{Q_i}\Psi_{\Lambda,Q}(y)|\leq\frac{C}{(1+|y-Q|)^3}.
\end{equation}

The above considerations take care of the linear term $\mu \epsilon^2 u$ in the equation but we still need to obtain approximate solutions which satisfy the boundary boundary condition. To this end,  we need an extra correction term. For this purpose, we define
 \begin{equation}
\label{2.8}
\hat{U}_{\Lambda,Q/\varepsilon}(z)=-\Psi_{\Lambda,Q/\e}(z)-c_n\mu^{-1}\e^{n-4}\Lambda^{\frac{n-2}{2}} H(\e z,Q)
+R_{\e,\Lambda, Q}(z)\chi(\e z),
\end{equation}
where $R_{\e,\Lambda, Q}$ is the unique solution satisfying the following boundary value problem
\begin{equation}
\label{2.9} \left\{\begin{array}{l}
 \Delta R_{\e,\Lambda, Q}-\e^2 R_{\e,\Lambda, Q}=0~~{\rm in }~~\Omega_{\e} \\
\mu\e^2\frac {\pt R_{\e,\Lambda, Q}}{\pt \nu}=-\frac {\pt}{\pt \nu}\Big{[}U_{\Lambda,Q/\e}-\mu\e^2\Psi_{\Lambda,Q/\e}-c_n\e^{n-2}\Lambda^{\frac {n-2}{2}}
 H(\e z,Q)\Big{]}~~{\rm on }~\partial\Omega_{\e}.
 \end{array}
 \right.
\end{equation}
Here $\chi(x)$~is a smooth cut-off function in $\Omega$~such that $\chi(x)=1$~for ~$d(x,\partial \Omega)< \delta/4$~and ~$\chi(x)=0$~for ~$d(x,\pt \Omega)>\delta/2.$

Observe that from (\ref{2.2}) and (\ref{2.6}), an expansion of $U_{\Lambda,Q/\e}$ and the definition of $H$ imply that the normal derivative of $R_{\e,Q}$ is of order $\e^{n-3}$ on the boundary of $\O,$ from which we deduce that\footnote{For $n=4$, the parameter $\Lambda$ is located in a range that depends on $\e$. Therefore, we have to take $\Lambda$ into consideration, and we note that each component on the right hand side of (\ref{2.9}) exactly carry $\Lambda$ as a factor.}
\begin{equation}
\label{2.10}
\big|R_{\e,\Lambda, Q}\big|+\big|\e^{-1}\nabla_{z}R_{\e,\Lambda, Q}\big|+\big|\e^{-2}\nabla^2_{z}R_{\e,\Lambda, Q}\big|\leq \begin{cases}C\Lambda,~~~~~~~~~~&~n=4,\\C\e^2,& ~n=6.\end{cases}
\end{equation}
Such an estimate also holds  for the derivatives of $R_{\e,\Lambda,Q}$~with respect to $\Lambda,~Q.$

Finally we are able to define the approximate bubble solutions.  Depending on the dimensions we shall use different ansatz.  For $n=4,$ we let
\begin{align}
\label{2.11}
\Lambda_{4,1} \leq \Lambda \leq \Lambda_{4,2},~Q\in \mathcal {M}_{\delta_4}:=\{x \in \Omega|~d(x,\partial \Omega)>\delta_4\},
\end{align}
where $\Lambda_{4,1}=\exp(-\frac12)\e^{\beta}$, $\Lambda_{4,2}=\exp(-\frac12)\e^{-\beta}$   are constants may depending on the domain and $\delta_4$ is a small constant, to be determined later. In viewing of the rescaling, we write
\begin{align*}
\bar{Q}=\frac {1}{\e}Q,
\end{align*}
and we define our approximate solutions as
\begin{equation}
\label{2.12}
W_{\e,\Lambda,Q}=U_{\Lambda,Q/\e}+\mu\e^2\hat{U}_{\Lambda,Q/\e}+\frac{c_4\Lambda}{|\Omega|}\mu^{-1}\e^{2}.
\end{equation}

For $n=6$, let $(\Lambda, Q, \eta)$ satisfy
\begin{align}
\label{2.13}
\sqrt{\frac{|\Omega|}{c_6}(\frac{1}{96}-\Lambda_6\e^{\frac23})}\leq&\Lambda
\leq\sqrt{\frac{|\Omega|}{c_6}(\frac{1}{96}+\Lambda_6\e^{\frac23})},\nonumber\\
Q\in \mathcal {M}_{\delta_6}:=\{x &\in \Omega|~d(x,\partial \Omega)>\delta_6\},\nonumber\\
\frac{1}{48}-\eta_6\e^{\frac13}\leq&\eta\leq\frac{1}{48}+\eta_6\e^{\frac13},
\end{align}
where $c_6= 4 |S^{5}|$, $\Lambda_6$ and $\eta_6$ are some constants that may depend on the domain, $\delta_6$ is a small constant, which is determined later. Our approximate solution for $n=6$ is the following
\begin{equation}
\label{2.14}
W_{\e,\Lambda,Q,\eta}=U_{\Lambda,\bar{Q}}+\mu\e^2\hat{U}_{\Lambda,\bar{Q}}+\eta\mu^{-1}\e^4.
\end{equation}

\medskip

We remark that unlike the case of $n=4$, in the case of $n=6$, {\em an extra parameter} $\eta$ is introduced.  The main reason is that when $n=6$ the linear term $-\mu \epsilon^2$ is lost in linearized outer problem. Actually this is one of the main difficulties. This seems to be quite {\em new} in the Neumann boundary value problems.

For convenience, in the following, we write $W,~U,~\hat{U},~R,~{\rm and}~\Psi$ instead of $W_{\e,\Lambda,Q}$, $U_{\e,Q/\e},$ $\hat{U}_{\Lambda,Q/\e}$, $R_{\e,\Lambda,Q}$ and $\Psi_{\Lambda,Q/\e}$ respectively. By construction, the normal derivative of $W$ vanishes on the boundary of $\Omega_{\e},$ and $W$ satisfies

\begin{align}
\label{2.15}
-\Delta W+\mu\e^2W=\left\{\begin{array}{ll}
8U^3+\mu^2\e^4\hat{U}-\mu\e^2\Delta (R_{\e,\Lambda,Q}\chi),~&n=4,\\
24U^2+\mu^2\e^4\hat{U}-\mu\e^2\Delta (R_{\e,\Lambda,Q}\chi)+\e^6(\eta-\frac{c_6\Lambda^2}{|\Omega|}),~&n=6.
\end{array}\right.
\end{align}

We note that $W$ depends smoothly on $\Lambda,~\bar{Q}.$~Setting, for $z\in \Omega_{\e},$
$$\l z-\bar{Q}\r=(1+|z-\bar{Q}|^2)^{\frac 12}.$$
A simple computation yields
\begin{align}
\label{2.16}
|W(z)|\leq&\left\{
\begin{array}{ll}
C(\e^2(-\ln\e)^{\frac{1}{2}}+\l z-\bar{Q}\r^{-2}),~&n=4,\\
C(\e^3+\l z-\bar{Q}\r^{-4}),~&n=6,
\end{array}\right.
\end{align}
\begin{align}
\label{2.17}
|D_{\Lambda}W(z)|\leq&\left\{
\begin{array}{ll}
C(\e^2(-\ln\e)^{\frac{1}{2}}+\l z-\bar{Q}\r^{-2}),~&n=4,\\
C\l z-\bar{Q}\r^{-4},~&n=6,
\end{array}\right.
\end{align}
\begin{align}
\label{2.18}
|D_{\bar{Q}}W(z)|\leq\left\{
\begin{array}{ll}
C(\l z-\bar{Q}\r^{-3}),~&n=4,\\
C(\l z-\bar{Q}\r^{-5}),~&n=6,
\end{array}\right.
\end{align}
and
\begin{equation}
\label{2.19}
|D_{\eta}W(z)|=O(\e^3),~n=6.
\end{equation}

According to the choice of $W,$~we have the following error and energy estimates, whose proof will be given in Section 7.

\begin{lemma}
\label{le2.1}
For $n=4$, we have
\begin{align}
\label{2.20}
\big|S_{\e}[W](z)\big|\leq~& C\Big{(} \l z-\bar{Q}\r^{-4}\e^2(-\ln\e)^{\frac{1}{2}}+\l z-\bar{Q}\r^{-2}\e^4(-\ln\e)\Big)\nonumber\\
&+C\Big(\big(\frac {\e^4}{(-\ln \e)^{\frac 12}}+\frac {\e^4}{(-\ln \e)}|\ln (\frac {1}{\e(1+|z-\bar{Q}|)})|\big)\Lambda\Big{)},
\end{align}
\begin{align}
\label{2.21}
\big|D_{\Lambda}S_{\e}[W](z)\big|\leq~& C\Big{(} \l z-\bar{Q}\r^{-4}\e^2(-\ln\e)^{\frac{1}{2}}+\l z-\bar{Q}\r^{-2}\e^4(-\ln\e)\nonumber\\
&+\frac {\e^4}{(-\ln \e)^{\frac 12}}+\frac {\e^4}{(-\ln \e)}|\ln (\frac {1}{\e(1+|z-\bar{Q}|)})|\Big{)},
\end{align}
\begin{align}
\label{2.22}
\big|D_{\bar{Q}}S_{\e}[W](z)\big|\leq~ &C\Big{(}\l z-\bar{Q}\r^{-5}\e^2(-\ln\e)^{\frac{1}{2}}+\l z-\bar{Q}\r^{-3}\e^4(-\ln\e)\nonumber\\
&+\l z-\bar{Q}\r^{-1}\frac{\e^4}{(-\ln\e)}\Big{)},
\end{align}
and
\begin{align}
\label{2.23}
J_{\e}[W]=&~2\int_{\Bbb R^4}U_{1,0}^4+\frac {c_4\Lambda^2}{4}\e^2(\frac {c_1}{-\ln \e})^{\frac 12}\ln \frac {1}{\Lambda \e}-\frac {c_4^2\Lambda^2}{2|\Omega|}\e^2(\frac {c_1}{-\ln \e})^{-\frac 12}\nonumber\\
&+\frac 12c_4^2\Lambda^2\e^2H(Q,Q)+O\big(\e^2\big(\frac{c_1}{-\ln\e}\big)^{\frac12}\Lambda^2\big)+O(\e^4(-\ln\e)^2).
\end{align}
For $n=6,$ we have
\begin{align}
\label{2.24}
S_{\e}[W](z)=-\e^6\big(24\eta^2-\eta+\frac{c_6\Lambda^2}{|\Omega|}\big)+O(1)\e^3\l z-\bar{Q}\r^{-4},
\end{align}
\begin{align}
\label{2.25}
\big|D_{\Lambda}S_{\e}[W](z)\big|=O(1)\big(\l z-\bar{Q}\r^{-4}\e^3+\e^6\big),
\end{align}
\begin{align}
\label{2.26}
\big|D_{\eta}S_{\e}[W](z)\big|=O(1)\big(\l z-\bar{Q}\r^{-4}\e^3+\e^{6\frac13}\big),
\end{align}
\begin{align}
\label{2.27}
\big|D_{\bar{Q}}S_{\e}[W](z)\big|\leq~& C\l z-\bar{Q}\r^{-5}\e^3,
\end{align}
and
\begin{align}
\label{2.28}
J_{\e}[W]=~&4\int_{\Bbb R^6}U_{1,0}^3+\Big(\frac12\eta^2|\Omega|-c_6\eta\Lambda^2+\frac{1}{48}c_6\Lambda^2-8\eta^3|\Omega|\Big)\e^3
+\frac12c_6^2\Lambda^4\e^4H(Q,Q)\nonumber\\&
+\frac12\big(\eta-\frac{c_6\Lambda^2}{|\Omega|}\big)\e^4\int_{\Omega}\frac{\Lambda^2}{|x-Q|^4}+O(\e^5).
\end{align}
\end{lemma}
\vspace{1cm}

\section{Finite-Dimensional Reduction}
\setcounter{equation}{0}

We now apply finite-dimensional reduction procedure for critical exponent problems. The original finite dimensional Liapunov-Schmidt reduction method was first introduced in a seminal paper by Floer and Weinstein \cite{FW} in their construction of single bump
solutions to one dimensional nonlinear Schrodinger equations. Subsequently this method has been modified and adapted to critical exponent problems. For critical exponents problems,
we refer to Bahri-Li-Rey \cite{BLR}, Del Pino-Felmer-Musso \cite{dfm}, Rey-Wei \cite{rw, rw2}
and Wei-Yan \cite{WY}  and the references therein.  For the most updated references and optimal treatments of finite dimensional reduction for critical problems, we refer to Li-Wei-Xu \cite{LWX}.

The general strategy of this method is as follows: the nonlinear equation (\ref{1.11}) is solved in two steps. In the first step, we solve it up to finite dimensional approximate kernels.  In the second step, we reduce the problem to finding critical points of a finite dimensional problems in a suitable sets.

The new element in our proof is in the case of $n=6$: an extra space (corresponding to $\eta$) is introduced. Unlike the traditional critical exponent problems, in which the dimensional of approximate kernels is $n+1$, we now have $n+2=8$ dimensions.

Equipping $H^1(\Omega_{\e})$ with the scalar product
\begin{equation}
\label{3.1}
(u,v)_{\e}=\int_{\Omega_{\e}}(\nabla u\cdot \nabla v+\mu\e^2uv).
\end{equation}

For the case $n=4,$ orthogonality to the functions
\begin{equation}
\label{3.2}
Y_0=\frac{\partial W}{\partial\Lambda},~Y_i=\frac{\partial W}{\partial \bar{Q}_i},1\leq i\leq 4,
\end{equation}
in that space is equivalent to the orthogonality in $L^2(\Omega_{\varepsilon}),$ equipped with the usual scalar product $\l\cdot,\cdot\r$, to
the functions $Z_i,0\leq i\leq 4$, defined as
\begin{equation}
\label{3.3}
\begin{cases}
Z_0=-\Delta \frac {\pt W}{\pt \Lambda}+\mu\e^2\frac {\pt W}{\pt \Lambda},\\
Z_i=-\Delta \frac {\pt W}{\pt \bar{Q}_i}+\mu\e^2\frac {\pt W}{\pt \bar{Q}_i},~1\leq i\leq 4.\\
\end{cases}
\end{equation}

Straightforward computations yield the estimate:

\begin{align}
\label{3.4}
|Z_i(z)|\leq C(\e^{4}+\l z-\bar{Q}\r^{-6}).
\end{align}

Then, we consider the following problem: given $h,$ finding a solution $\phi$ which satisfies
\begin{equation}
\label{3.5}
\begin{cases}
-\Delta \phi+\mu\e^2\phi-24W^{2}\phi=h+\Sigma_{i=0}^4 c_i Z_i~& {\rm in }~\Omega_{\e},\\
\frac {\pt \phi}{\pt \nu}=0 &{\rm on }~\pt \Omega_{\e},\\
\l Z_i,\phi\r=0,~& 0\leq i\leq 4,
\end{cases}
\end{equation}
for some numbers $c_i.$

While for the case $n=6,$ orthogonality to the functions
\begin{equation}
\label{3.6}
Y_0=\frac{\partial W}{\partial\Lambda},~Y_i=\frac{\partial W}{\partial \bar{Q}_i},1\leq i\leq 6,~Y_7=\frac{\partial W}{\partial\eta},
\end{equation}
in that space is equivalent to the orthogonality in $L^2(\Omega_{\varepsilon}),$ equipped with the usual scalar product $\l\cdot,\cdot\r$, to
the functions $Z_i,0\leq i\leq 7$, defined as
\begin{equation}
\label{3.7}
\begin{cases}
Z_0=-\Delta \frac {\pt W}{\pt \Lambda}+\mu\e^2\frac {\pt W}{\pt \Lambda},\\
Z_i=-\Delta \frac {\pt W}{\pt \bar{Q}_i}+\mu\e^2\frac {\pt W}{\pt \bar{Q}_i},~1\leq i\leq 6,\\
Z_7=-\Delta \frac {\pt W}{\pt \eta}+\mu\e^2\frac {\pt W}{\pt \eta}.
\end{cases}
\end{equation}
Direct computations give the following estimate:
\begin{align}
\label{3.8}
|Z_i(z)|\leq C(\e^{6}+\l z-\bar{Q}\r^{-8}),~0\leq i\leq 6,~Z_7(z)=O(\e^6).
\end{align}

Then, we consider the following problem: given $h,$ finding a solution $\phi$ which satisfies
\begin{equation}
\label{3.9}
\begin{cases}
-\Delta \phi+\mu\e^2\phi-48W\phi=h+\Sigma_{i=0}^7 d_i Z_i~& {\rm in }~\Omega_{\e},\\
\frac {\pt \phi}{\pt \nu}=0 &{\rm on }~\pt \Omega_{\e},\\
\l Z_i,\phi\r=0,~& 0\leq i\leq 7,
\end{cases}
\end{equation}
for some numbers $d_i.$

Existence and uniqueness of $\phi$  will follow from an inversion procedure in suitable weighted function space. To this end, we define
\begin{equation}
\label{3.10}
\begin{cases}
\|\phi\|_{*}=\|\l z-\bar{Q}\r\phi(z)\|_{\infty},~\|f\|_{**}=\e^{-3}(-\ln \e)^{\frac 12}|\overline{f}|+\|\l z-\bar{Q}\r^{3}f(z)\|_{\infty},~ n=4,\\
\|\phi\|_{***}=\|\l z-\bar{Q}\r^{2}\phi(z)\|_{\infty},~\|f\|_{****}=\|\l z-\bar{Q}\r^{4}f(z)\|_{\infty},~ n=6,
\end{cases}
\end{equation}
where $\|f\|_{\infty}=
\max_{z\in\Omega_{\varepsilon}}|f(z)|$ and $\overline{f}=|\Omega_{\varepsilon}|^{-1}\int_{\Omega_{\varepsilon}}f(z)\mathrm{d}z$ denotes the average of $f$ in $\Omega_{\varepsilon}.$

Before stating an existence result for $\phi$ in (\ref{3.5}) and (\ref{3.9}),  we need the following lemma:
\begin{lemma}
\label{le3.1}
Let $u$ and $f$ satisfy
$$-\Delta u=f,~\frac {\pt u}{\pt \nu}=0,~~\bar{u}=\bar{f}=0.$$
Then
\begin{equation}
\label{3.11}
|u(x)|\leq C\int_{\Omega_{\e}}\frac {|f(y)|}{|x-y|^{n-2}}{\rm d}y.
\end{equation}
\end{lemma}
\begin{proof}
The proof is similar to Lemma 3.1 in \cite{rw}, we omit it here.
\end{proof}
As a consequence, we have

\begin{corollary}
\label{cr3.1}
For $n=4,$ suppose $u$ and $f$ satisfy
$$-\Delta u+\mu\e^2u=f~~{\rm in }~\Omega_{\e},~~\frac {\pt u}{\pt \nu}=0~~{\rm on }~\pt\Omega_{\e}.$$
Then
\begin{equation}
\label{3.12}
\|u\|_{*}\leq C\|f\|_{**}.
\end{equation}
For $n=6,$ suppose $u$ and $f$ satisfy
$$-\Delta u+c\mu\e^2u=f~~{\rm in }~\Omega_{\e},~~\frac {\pt u}{\pt \nu}=0~~{\rm on }~\pt\Omega_{\e},~~\overline{u}=\overline{f}=0,$$
where $c$ is an arbitrary constant. Then
\begin{equation}
\label{3.13}
\|u\|_{***}\leq C\|f\|_{****}.
\end{equation}
\end{corollary}

\begin{proof}
For $n=4,$ integrating the equation yields $\bar{f}=\mu\e^2\bar{u}.$ We may rewrite the original equation as
$$\Delta (u-\bar{u})=\mu\e^2(u-\bar{u})-(f-\bar{f}).$$
With the help of Lemma \ref{le3.1}, we get
$$|u(y)-\bar{u}|\leq C\mu\e^2\int_{\Omega_{\e}}\frac {|u(x)-\bar{u}|}{|x-y|^2}{\rm d}x+C\int_{\Omega_{\e}}\frac {|f(x)-\bar{f}|}{|x-y|^2}{\rm d}x.$$
Since $$\l y-\bar{Q}\r\int_{\mathbb{R}^4}\frac{1}{|x-y|^{2}}\langle x-\bar{Q}\rangle^{-3}\mathrm{d}x<\infty,$$
we obtain
\begin{align*}
\|\l y-\bar{Q}\r|u-\bar{u}|\|_{\infty} &\leq C \mu\e^2\|\l y-\bar{Q}\r^{3}|u-\bar{u}|\|_{\infty}+C\|\l y-\bar{Q}\r^{3}|f-\bar{f}|\|_{\infty}\\&\leq C
\mu\|\l y-\bar{Q}\r|u-\bar{u}|\|_{\infty}+C\|\l y-\bar{Q}\r^{3}|f-\bar{f}|\|_{\infty},\end{align*}
which gives
$$\|\l y-\bar{Q}\r|u-\bar{u}|\|_{\infty} \leq C\|\l y-\bar{Q}\r^{3}|f-\bar{f}|\|_{\infty},$$
whence
\begin{align*}
\|\l y-\bar{Q}\r u\|_{\infty} \leq C\|\l y-\bar{Q}\r\|_{\infty}|\bar{u}|+C\e^{-3}|\bar{f}|+\|\l y-\bar{Q}\r^{3}f\|_{\infty}\leq C\|f\|_{**}.
\end{align*}
Hence we finish the proof of the case $n=4.$

For $n=6,$ by the help of Lemma \ref{le3.1},
\begin{align*}
|\l y-\bar{Q}\r^2 u|\leq C\int_{\Omega_{\e}} \frac{\l y-\bar{Q}\r^2(|\mu\e^2u|+|f|)}{|x-y|^4}\mathrm{d}x
\leq C(|\mu\ln\e|\|u\|_{***}+\|f\|_{****}),
\end{align*}
where we used some similar estimates appeared in $n=4.$ From the above inequality, we obtain $\|u\|_{***}\leq\|f\|_{****}.$ Hence we finish the proof.
\end{proof}

We now state the main result in this section.
\begin{proposition}
\label{pr3.1}
There exists $\varepsilon_0>0$ and a constant $C>0$, independent of $\varepsilon,~\Lambda,~Q$ satisfying (\ref{2.11}) and independent of $\e,~\eta,~\Lambda,~Q$ satisfying (\ref{2.13}), such that for all $0<\varepsilon<\varepsilon_0$ and all $h\in L^{\infty}(\Omega_{\varepsilon})$, problem (\ref{3.5}) and (\ref{3.9}) has a unique solution $\phi=L_{\varepsilon}(h)$. Furthermore, for equation (\ref{3.5}) and (\ref{3.9}), we have the following estimates,
\begin{align}
\label{3.14}
&\|L_{\varepsilon}(h)\|_{*}\leq C\|h\|_{**},~|c_{i}|\leq C\|h\|_{**}~\mathrm{for}~0\leq i\leq 4,\nonumber\\
&\|L_{\varepsilon}(h)\|_{***}\leq C\|h\|_{****},~|d_{i}|\leq C\|h\|_{****}~\mathrm{for}~0\leq i\leq 6.
\end{align}
Moreover, the map $L_{\varepsilon}(h)$ is $C^1$ with respect to $\Lambda,~\bar{Q}$ of the $L^{\infty}_*$-norm in $n=4$ and with respect to $\Lambda,~\bar{Q},~\eta$ of the $L^{\infty}_{***}$-norm in $n=6$, i.e.,
\begin{equation}
\label{3.15}
\|D_{(\Lambda,\bar{Q})}L_{\varepsilon}(h)\|_*\leq C\|h\|_{**}~\mathrm{in}~n=4,\quad\|D_{(\eta,\Lambda,\bar{Q})}L_{\varepsilon}(h)\|_{***}\leq C\varepsilon^{-1}\|h\|_{****}~\mathrm{in}~n=6.
\end{equation}
\end{proposition}

The argument goes the same as the Proposition 3.1 in \cite{rw}, for convenience of the reader, we sketch the proof here. First, we need the following Lemma.

\begin{lemma}
\label{le3.2}
For $n=4$, assume that $\phi_{\varepsilon}$ solves (\ref{3.5}) for $h=h_{\varepsilon}$. If $\|h_{\varepsilon}\|_{**}$ goes to zero as $\varepsilon$ goes to zero, so does $\|\phi_{\varepsilon}\|_{*}.$ While for $n=6,$ assume that $\phi_{\varepsilon}$ solves (\ref{3.9}) for $h=h_{\varepsilon}$. If $\|h_{\varepsilon}\|_{****}$ goes to zero as $\varepsilon$ goes to zero, so does $\|\phi_{\varepsilon}\|_{***}.$
\end{lemma}

\begin{proof}
We prove this lemma by contradiction. First we consider for the case $n=4$. Assume that $\|\phi_{\e}\|_{*}=1$. Multiplying the first equation in (\ref{3.5}) by $Y_j$ and integrating in $\O$ we find
\begin{equation*}
\sum_{i}c_i\l Z_i,Y_j\r=\langle -\Delta Y_j+\mu\e^2Y_j-24W^2Y_j,\phi_{\varepsilon}\rangle-
\langle h_{\varepsilon},Y_j\rangle.
\end{equation*}
We can easily get the following equalities from the definition of $Z_i,Y_j$
\begin{align}
\label{3.16}
&\langle Z_0,Y_0\rangle=\|Y_0\|_{\varepsilon}^2=\gamma_0+o(1),\nonumber\\
&\langle Z_i,Y_i\rangle=\|Y_i\|_{\varepsilon}^2=\gamma_1+o(1),~1\leq i\leq 4,
\end{align}
where $\gamma_0,\gamma_1$ are strictly positive constants, and
\begin{equation}
\label{3.17}
\langle Z_i,Y_j\rangle=o(1),~i\neq j.
\end{equation}
On the other hand, in view of the definition of $Y_j$ and $W$, straightforward computations yield
$$\langle-\Delta Y_j+\mu\e^2Y_j-24W^{2}Y_j,\phi_{\varepsilon}\rangle=o(\|\phi_{\varepsilon}\|_*)$$
and
$$\langle h_{\varepsilon},Y_j\rangle=O(\|h_{\varepsilon}\|_{**}).$$
Consequently, inverting the quasi diagonal linear system solved by the $c_i$'s we find
\begin{equation}
\label{3.18}
c_i=O(\|h_{\varepsilon}\|_{**})+o(\|\phi_{\varepsilon}\|_*).
\end{equation}
In particular, $c_i=o(1)$ as $\varepsilon$ goes to zero.

Since $\|\phi_{\varepsilon}\|_*=1$, elliptic theory shows that along some subsequence, the functions $\phi_{\e,0}=\phi_{\varepsilon}(y-\bar{Q})$ converge uniformly in any compact subset of $\mathbb{R}^4$ to a nontrivial solution of
$$-\Delta\phi_0=24U_{\Lambda,0}^{2}\phi_0.$$
A bootstrap argument (see e.g. Proposition 2.2 of \cite{ww}) implies $|\phi_0(y)|\leq C(1+|y|)^{-2}.$ As consequence, $\phi_0$ can be written as
$$\phi_0=\alpha_0\frac{\partial U_{\Lambda,0}}{\partial \Lambda}+\sum_{i}\alpha_i\frac{\partial U_{\Lambda,0}}{\partial y_i}$$
(see \cite{r1}). On the other hand, equalities $\langle Z_i,\phi_{\varepsilon}\rangle=0$ yield
\begin{align*}
&\int_{\mathbb{R}^4}-\Delta\frac{\partial U_{\Lambda,0}}{\partial \Lambda}\phi_0=\int_{\mathbb{R}^4}U_{\Lambda,0}^{2}\frac{\partial U_{\Lambda,0}}{\partial \Lambda}\phi_0=0,\\
&\int_{\mathbb{R}^4}-\Delta\frac{\partial U_{\Lambda,0}}{\partial y_i}\phi_0=\int_{\mathbb{R}^4}U_{\Lambda,0}^{2}\frac{\partial U_{\Lambda,0}}{\partial y_i}\phi_0=0,~1\leq i\leq 4.
\end{align*}
As we also have
$$\int_{\mathbb{R}^4}\Big|\nabla\frac{\partial U_{\Lambda,0}}{\partial \Lambda}\Big|^2=\gamma_0>0,~
\int_{\mathbb{R}^4}\Big|\nabla\frac{\partial U_{\Lambda,0}}{\partial y_i}\Big|^2=\gamma_1>0,~1\leq i\leq 4,$$
and
$$\int_{\mathbb{R}^4}\nabla\frac{\partial U_{\Lambda,0}}{\partial \Lambda}\nabla\frac{\partial U_{\Lambda,0}}{\partial y_i}=\int_{\mathbb{R}^4}\nabla\frac{\partial U_{\Lambda,0}}{\partial y_i}\nabla\frac{\partial U_{\Lambda,0}}{\partial y_j}=0,~i\neq j,$$
the $\alpha_i's$ solve a homogeneous quasi diagonal linear system, yielding $\alpha_i=0,0\leq i\leq 4,$ and $\phi_0=0.$ So $\phi_{\varepsilon}(z-\bar{Q})\rightarrow0$ in $C^1_{loc}(\Omega_{\varepsilon}).$

Next, we will show $\|\phi_{\e}\|_*=o(1)$ by using the equation (\ref{3.5}). Using (\ref{3.5}) and Corollary \ref{cr3.1}, we have
\begin{equation}
\label{3.19}
\|\phi_{\varepsilon}\|_*\leq C(\|W^{2}\phi_{\varepsilon}\|_{**}+\|h\|_{**}+\sum_i|c_i|\|Z_i\|_{**}).
\end{equation}
Then we estimate the right hand side of (\ref{3.19}) term by term. By the help of (\ref{2.16}), we deduce that
\begin{equation}
\label{3.20}
|\l z-\bar{Q}\r^{3}W^2\phi_{\e}|\leq C\e^4(-\ln \e)\l z-\bar{Q}\r^2\|\phi_{\e}\|_{*}+\l z-\bar{Q}\r^{-1}|\phi_{\e}|.
\end{equation}
Since $\|\phi_{\e}\|_{*}=1,$ the first term on the right hand side of (\ref{3.20}) is dominated by $\e^2(-\ln \e).$ The last term goes uniformly to zero in any ball $B_{R}(\bar {Q}),$ and is dominated by $\l z-\bar{Q}\r^{-2}\|\phi_{\e}\|_{*}=\l z-\bar{Q}\r^{-2},$ which, through the choice of $R,$ can be made as small as possible in $\Omega_{\e}\backslash B_{R}(\bar {Q}).$ Consequently,
\begin{equation}
\label{3.21}
|\l z-\bar{Q}\r^{3}W^2\phi_{\e}|=o(1)
\end{equation}
as $\e$ goes to zero, uniformly in $\Omega_{\e}.$ On the other hand, we can also get
\begin{align*}
\e^{-3}(-\ln \e)^{\frac 12}{\overline {W^2\phi_{\e}}} & \leq C\e (-\ln \e)^{\frac 12}\int_{\Omega_{\e}}(\l z-\bar{Q}\r^{-4}+\e^4(-\ln \e))|\phi_{\e}|\\
& \leq C\e(-\ln \e)^{\frac 12}\int_{\Omega_{\e}}(\l z-\bar{Q}\r^{-5}+\e^4(-\ln \e)\l z-\bar{Q}\r^{-1})\|\phi_{\e}\|_{*}\\&= o(1).
\end{align*}
Finally, we obtain
$$\|W^2\phi_{\e}\|_{**}=o(1).$$
In view of the formula (\ref{3.4}), we have
$$\l z-\bar{Q}\r^3|Z_i| \leq C(\l z-\bar{Q}\r^{3}\e^4(\frac{1}{-\ln\e})+\l z-\bar{Q}\r^{-3})= O(1).$$
and
$$\e^{-3}(-\ln \e)^{\frac {1}{2}}\overline {Z_i}\leq C \e(-\ln \e)^{\frac {1}{2}}\int_{\O}|\l z-\bar{Q}\r^{-6}+\e^4|{\rm d}x=o(1).$$
Hence, $\|Z_i\|_{**}=O(1).$ Therefore, we have
\begin{equation}
\label{3.22}
\|\phi_{\e}\|_*\leq C(\|W^2\phi_{\e}\|_{**}+\|h\|_{**}+\sum_i|c_i|\|Z_i\|_{**})=o(1),
\end{equation}
which contradicts our assumption that $\|\phi_{\e}\|_*=1$.
\medskip

For $n=6.$ We still assume that $\|\phi_{\e}\|_{***}=1.$ Using the similar arguments in previous case, we obtain the following
\begin{equation}
\label{3.23}
d_i=O(\|h\|_{****})+o(\|\phi\|_{***})~\mathrm{for}~0\leq i\leq6,~d_7=O(\e^{-2}\|h\|_{****})+ O(\e^{-1}\|\phi\|_{***}).
\end{equation}
and $\phi_{\e}(z-\bar{Q})\rightarrow0$ in $C^1_{loc}({\Omega_{\e}}).$ Next, we will show $\|\phi_{\e}\|_{***}=o(1)$ by using the equation (\ref{3.9}). At first, we write the equation (\ref{3.9}) into the following
\begin{equation}
\label{3.24}
-\Delta\phi_{\e}+\mu\e^2(1-48\eta)\phi_{\e}=h+\sum_id_iZ_i+48U\phi_{\e}+48\e^3\hat{U}\phi_{\e}.
\end{equation}
Since $\int_{\O}\phi=0,$ as a result, we can find the integral for both sides of (\ref{3.24}) in $\O$ are $0.$ Using Corollary \ref{cr3.1} again, we have
\begin{equation}
\label{3.25}
\|\phi_{\e}\|_{***}\leq C(\|(U+\e^3\hat{U})\phi_{\e}\|_{****}+\|h\|_{****}+\sum_{i}|d_i|\|Z_i\|_{****}).
\end{equation}
From the formula of $U$ and $\hat{U},$ it is not difficult to show
$$U+\e^3\hat{U}\leq C\l z-\bar{Q}\r^{-4}.$$
Similar to the case $n=4,$ we could show $\|\l z-\bar{Q}\r^{-4}\phi_{\e}\|_{****}=o(1)$, $\|Z_i\|_{****}=O(1)$ for $0\leq i\leq6$ and $\|Z_7\|_{****}=O(\e^2).$ Therefore, by the above facts and (\ref{3.23}), we conclude
$$\|\phi_{\e}\|_{***}\leq o(1)+C\|h\|_{****}+o(1)\|\phi_{\e}\|_{***}=o(1)$$
which contradicts the previous assumption that $\|\phi_{\e}\|_{***}=1.$ Hence, we finish the proof.
\end{proof}

\noindent {\it Proof of Proposition \ref{pr3.1}.} Since the proof of the case $n=4$ and $n=6$ are almost the same, we only give the proof for the former one. We set
$$H=\{\phi\in H^1(\Omega_{\varepsilon})\mid\langle Z_i,\phi\rangle=0,~0\leq i\leq 4\},$$
equipped with the scalar product $(\cdot,\cdot)_{\varepsilon}.$ Problem (\ref{3.5}) is equivalent to find $\phi\in H$ such that
$$(\phi,\theta)_{\varepsilon}=\langle 24W^{2}\phi+h,\theta\rangle,\quad\forall \theta\in H,$$
that is
\begin{align}
\label{3.26}
\phi=T_{\varepsilon}(\phi)+\tilde{h},
\end{align}
where $\tilde{h}$ depends on $h$ linearly, and $T_{\varepsilon}$ is a compact operator in $H.$ Fredholm's alternative
ensures the existence of a unique solution, provided that the kernel of $Id-T_{\varepsilon}$ is reduced to $0$. We notice that any $\phi_{\varepsilon}\in \mathrm{Ker}(Id-T_{\varepsilon})$ solves (\ref{3.5}) with $h=0.$ Thus, we deduce from Lemma \ref{le3.2} that $\|\phi_{\varepsilon}\|_*=o(1)$ as $\varepsilon$ goes to zero. As $\mathrm{Ker}(Id-T_{\varepsilon})$ is a vector space and is $\{0\}.$ The inequalities (\ref{3.14}) follow from Lemma \ref{le3.2} and (\ref{3.18}). This completes the proof of the first part of Proposition \ref{pr3.1}.

The smoothness of $L_{\varepsilon}$ with respect to $\Lambda$ and $\bar{Q}$ is a consequence of the smoothness of $T_{\varepsilon}$ and $\tilde{h}$, which occur in the implicit definition (\ref{3.26}) of $\phi\equiv L_{\varepsilon}(h)$, with respect to these variables. Inequality (\ref{3.15}) is obtained by differentiating (\ref{3.5}), writing the derivatives of $\phi$ with respect $\Lambda$ and $\bar{Q}$ as linear combinations of the $Z_i$'s and an orthogonal part, and estimating each term by using the first part of the proposition, one can see \cite{dfm},\cite{kr} for detailed computations. \hfill $\Box$
\vspace{1cm}

\section{Finite-dimensional reduction:a nonlinear problem}
\setcounter{equation}{0}
In this section, we turn our attention to the nonlinear problem, which we solve in the finite-dimensional subspace orthogonal to the $Z_i.$~Let $S_{\e}[u]$ be as defined at $(\ref{1.12}).$ Then (\ref{1.11}) is equivalent to

\begin{equation}
\label{4.1}
S_{\varepsilon}[u]=0~\mathrm{in}~\Omega_{\varepsilon},\quad u_+\neq0,\quad\frac{\partial u}{\partial\nu}=0~\mathrm{on}~\partial\Omega_{\varepsilon}.
\end{equation}
Indeed, if $u$ satisfies (\ref{4.1}), the Maximal Principle ensures that $u>0$ in $\Omega_{\varepsilon}$ and (\ref{1.12}) is satisfied. Observing that
$$S_{\e}[W+\phi]=-\Delta (W+\phi)+\mu\e^2(W+\phi)-n(n-2)(W+\phi)^{\frac{n+2}{n-2}}$$
may be written as
\begin{equation}
\label{4.2}
S_{\e}[W+\phi]=-\Delta \phi+\mu\e^2\phi-n(n+2)W^{\frac {4}{n-2}}\phi+R^{\e}-n(n-2)N_{\e}(\phi)
\end{equation}
with
\begin{equation}
\label{4.3}
N_{\e}(\phi)=(W+\phi)^{\frac{n+2}{n-2}}-W^{\frac{n+2}{n-2}}-\frac{n+2}{n-2}W^{\frac{4}{n-2}}\phi
\end{equation}
and
\begin{equation}
\label{4.4}
R^{\e}=S_{\e}[W]=-\Delta W+\mu\e^2W-n(n-2)W^{\frac{n+2}{n-2}}.
\end{equation}\par

From Lemma \ref{le2.1} we get
\begin{equation}
\label{4.5}
\left\{\begin{array}{ll}
\|R^{\e}\|_{**}\leq C\e\Lambda+\e^2(-\ln\e)^{\frac12},~\|D_{{(}\Lambda,\bar{Q}{)}}R^{\e}\|_{**}\leq C\e,~&n=4,\\
\|R^{\e}\|_{****}\leq C\e^{2\frac{2}{3}},~\|D_{{(}\Lambda,\bar{Q},\eta{)}}R^{\e}\|_{****}\leq C\e^{2},~&n=6.
\end{array}\right.
\end{equation}

We now consider the following nonlinear problem: finding $\phi$~such that, for some numbers $c_i$,
\begin{equation}
\label{4.6}
\begin{cases}
-\Delta (W+\phi)+\mu\e^2(W+\phi)-8(W+\phi)^{3}=\sum_i c_iZ_i~~&{\rm in }~\Omega_{\e},\\
\frac {\pt \phi}{\pt \nu}=0~~&{\rm on}~\partial \Omega_{\e},\\
\l Z_i,\phi\r=0,~~~ &0 \leq i\leq4\end{cases}
\end{equation}
for $n=4,$ and finding $\phi$~such that, for some numbers $d_i$,
\begin{equation}
\label{4.7}
\begin{cases}
-\Delta (W+\phi)+\mu\e^2(W+\phi)-24(W+\phi)^{2}=\sum_i d_iZ_i~~&{\rm in }~\Omega_{\e},\\
\frac {\pt \phi}{\pt \nu}=0~~&{\rm on}~\partial \Omega_{\e},\\
\l Z_i,\phi\r=0,~~~ &0 \leq i\leq7\end{cases}
\end{equation}
for $n=6.$ The first equation in (\ref{4.6}) and (\ref{4.7}) can be also written as
\begin{align}
\label{4.8}
-\Delta \phi+\mu\e^2\phi-24W^2\phi=8N_{\e}(\phi)-R^{\e}+\sum_i{c_i}Z_i,\nonumber\\
-\Delta \phi+\mu\e^2\phi-48W\phi=24N_{\e}(\phi)-R^{\e}+\sum_i{d_i}Z_i.
\end{align}
In order to employ the contraction mapping theorem to prove that (\ref{4.6}) and (\ref{4.7}) are uniquely solvable in the set where $\|\phi\|_{*}$ and $\|\phi\|_{***}$ are small respectively, we need to estimate $N_{\e}$ in the following lemma.

\begin{lemma}
\label{le4.1}
There exists $\varepsilon_1>0,$ independent of $\Lambda,\bar{Q},\eta$ and $C$ independent of $\varepsilon,\Lambda,\bar{Q},\eta$ such that for $\varepsilon\leq\varepsilon_1$ and
\begin{align*}
\|\phi\|_*\leq C\e\Lambda~\mathrm{for}~n=4,\quad\|\phi\|_{***}\leq C\e^{2\frac23}~\mathrm{for}~n=6.
\end{align*}
Then,
\begin{equation}
\label{4.9}
\|N_{\varepsilon}(\phi)\|_{**}\leq C\e\Lambda\|\phi\|_*~\mathrm{for}~n=4,\quad\|N_{\varepsilon}(\phi)\|_{****}\leq C\e\|\phi\|_{***}~\mathrm{for}~n=6.
\end{equation}
For
\begin{align*}
\|\phi_i\|_{*}\leq C\e\Lambda~\mathrm{for}~n=4,\quad\|\phi_i\|_{***}\leq C\e^{2\frac23}~\mathrm{for}~n=6,~~\quad i=1,2.
\end{align*}
Then,
\begin{align}
\label{4.10}
&\|N_{\varepsilon}(\phi_1)-N_{\varepsilon}(\phi_2)\|_{**}\leq C\e\Lambda\|\phi_1-\phi_2\|_*~\mathrm{for}~n=4,\nonumber\\
&\|N_{\varepsilon}(\phi_1)-N_{\varepsilon}(\phi_2)\|_{****}\leq C\e\|\phi_1-\phi_2\|_{***}~\mathrm{for}~n=6.
\end{align}
\end{lemma}

\begin{proof}
Since the proof of these two cases are similar, we only consider $n=4$ here.
From (\ref{4.3}), we see that
\begin{equation}
\label{4.11}
|N_{\e}(\phi)|\leq C(W\phi^2+|\phi|^3).
\end{equation}
Using (\ref{2.15}), we infer that
\begin{align*}
\e^{-3}(-\ln \e)^{\frac 12}\overline{W\phi^2+|\phi|^3}=
\e(-\ln \e)^{\frac 12}\int_{\Omega_{\e}}(W\phi^2+|\phi|^3),
\end{align*}
where the integration term on the right hand side of the above equality can be estimated as
\begin{align*}
\big|W\phi^2+|\phi|^3\big|\leq ~&C\left(\big(\l z-\bar{Q}\r^{-2}+\e^2(-\ln\e)^{\frac{1}{2}}\big)|\phi|^2+|\phi|^3\right)\\
\leq ~&C\Big(\big(\l z-\bar{Q}\r^{-4}+\e^2(-\ln\e)^{\frac{1}{2}}\l z-\bar{Q}\r^{-2}\big)\|\phi\|_*^2+\l z-\bar{Q}\r^{-3}\|\phi\|_*^3\Big)\\
\leq ~&C\Big(\big(\e\l z-\bar{Q}\r^{-4}+\e^3(-\ln\e)^{\frac{1}{2}}\l z-\bar{Q}\r^{-2}\big)\Lambda\Big)\|\phi\|_*.
\end{align*}
As a consequence,
\begin{align*}
\e^{-3}(-\ln \e)^{\frac 12}\overline{W\phi^2+|\phi|^3}\leq C\e^2(-\ln\e)^{\frac{3}{2}}
\Lambda\|\phi\|_*\leq C\e\Lambda\|\phi\|_*.
\end{align*}
On the other hand,
$$\|\l z-\bar{Q}\r^{3}(W\phi^2+|\phi|^3)\|_{\infty}\leq C\e\Lambda\|\phi\|_{*}.$$
Thus, $(\ref{4.9})$ follows. Concerning (\ref{4.10}), we write
$$N_{\e}(\phi_1)-N_{\e}(\phi_2)=\pt_{\vartheta}N_{\e}(\vartheta)(\phi_1-\phi_2)$$
for some
$\vartheta=x\phi_1+(1-x)\phi_2,~x\in[0,1].$ From
$$\pt_{\vartheta}N_{\e}(\vartheta)=3[(W+\vartheta)^2-W^2],$$
we deduce that
\begin{equation}
\label{4.12}
\pt_{\vartheta}N_{\e}(\vartheta)\leq C(|W||\vartheta|+\vartheta^2)
\end{equation}
and the proof of (\ref{4.10}) is similar to the previous one.
\end{proof}

\begin{proposition}
\label{pr4.1}
For the case $n=4,$ there exists $C,~$independent of $\e$ and $\Lambda,~Q$ satisfying (\ref{2.11}), such that  for small $\e$ problem (\ref{4.6}) has a unique solution  $\phi=\phi(\Lambda,\bar{Q},\e)$ with
\begin{align}
\label{4.13}
\|\phi\|_{*}\leq  C\e\Lambda.
\end{align}
Moreover, $(\Lambda,\bar{Q})\rightarrow \phi(\Lambda,\bar{Q},\e)$ is $C^1$ with respect to the $*$-norm, and
\begin{align}
\label{4.14}
\|D_{(\Lambda,\bar{Q})}\phi\|_{*}\leq  C\e.
\end{align}
For the case $n=6,$ there exists $C,~$independent of $\e$ and $\Lambda,~\eta,~Q$ satisfying (\ref{2.13}), such that  for small $\e$ problem (\ref{4.7}) has a unique solution  $\phi=\phi(\Lambda,\eta,\bar{Q},\e)$ with
\begin{align}
\label{4.15}
\|\phi\|_{***}\leq  C\e^{\frac{8}{3}}.
\end{align}
Moreover, $(\Lambda,\eta,\bar{Q})\rightarrow \phi(\Lambda,\eta,\bar{Q},\e)$ is $C^1$ with respect to the $***$-norm, and
\begin{align}
\label{4.16}
\|D_{(\Lambda,\eta,\bar{Q})}\phi\|_{***}\leq  C\e^{\frac{5}{3}}.
\end{align}
\end{proposition}

\begin{proof}
We only give the proof of $n=4$, the other case can be argued similarly. In the same spirit of  \cite{dfm}, we consider the map $A_\e$ from $\cal{F}$=$\{\phi\in H^1(\Omega_{\varepsilon})|\|\phi\|_{*}\leq C^{'}\e\Lambda\}$ to $H^1(\Omega_{\varepsilon})$ defined as
$$A_{\e}(\phi)=L_{\e}(8N_{\e}(\phi)+R^{\e}).$$
Here $C^{'}$~is a large number, to be determined later, and $L_{\e}$~is given by Proposition \ref{pr3.1}. We note that finding a solution $\phi$ to problem (\ref{4.6}) is equivalent to finding a fixed point of $A_{\e}.$ On the one hand, we have for $\phi\in \cal{F}$, using (\ref{4.5}), Proposition \ref{pr3.1} and Lemma \ref{le4.1},
\begin{align*}
\|A_{\e}(\phi)\|_{*}&\leq 8\|L_{\e}(N_{\e}(\phi))||_{*}+\|L_{\e}(R^{\e})\|_{*}\leq C_1(\|N_{\e}(\phi)\|_{**}+\e\Lambda)\\&\leq C_2C^{'}\e^{2}\Lambda+C_1\e\Lambda\leq C^{'}\e\Lambda
\end {align*}
for $C^{'}=2C_1$ and $\e$ small enough, implying that $A_{\e}$ sends $\cal{F}$ into itself. On the other hand, $A_{\e}$ is a contraction. Indeed, for $\phi_1~{\rm and}~\phi_2~{\rm in }~\cal{F},$ we write
$$\|A_{\e}(\phi_1)-A_{\e}(\phi_2)\|_{*}\leq C\|N_{\e}(\phi_1)-N_{\e}(\phi_2)\|_{**}\leq C\e\Lambda\|\phi_{1}-\phi_2\|_{*}\leq \frac 12\|\phi_{1}-\phi_2\|_{*}$$
for $\e$ small enough. The contraction Mapping Theorem implies that $A_{\e}$ has a unique fixed point in $\cal{F},$ that is, problem (\ref{4.6}) has a unique solution $\phi$ such that $\|\phi\|_{*}\leq C^{'}\e\Lambda.$

In order to prove that $(\Lambda,\bar{Q})\rightarrow \phi(\Lambda,\bar{Q})$ is $C^1,$ we remark that if we set for $\psi\in F,$
$$B(\Lambda,\bar{Q},\psi)\equiv \psi-L_{\e}(8N_{\e}(\psi)+R^{\e}),$$
then $\phi$ is defined as
\begin{equation}
\label{4.17}
B(\Lambda,\bar{Q},\phi)=0.
\end{equation}
We have
$$\partial_{\psi}B(\Lambda,\bar{Q},\psi)[\theta]=\theta-8L_{\varepsilon}(\theta(\partial_{\psi}N_{\varepsilon})
(\psi)).$$
Using Proposition \ref{pr3.1} and (\ref{4.12}) we write
\begin{align*}
\|L_{\varepsilon}(\theta(\partial_{\psi}N_{\varepsilon})
(\psi))\|_*\leq&~C\|\theta(\partial_{\psi}N_{\varepsilon})
(\psi)\|_{**}\leq\|\langle z-\bar{Q}\rangle^{-1}(\partial_{\psi}N_{\varepsilon})
(\psi)\|_{**}\|\theta\|_*\\
\leq&~C\|\langle z-\bar{Q}\rangle^{-1}(W_+|\psi|+|\psi|^2)\|_{**}\|\theta\|_*.
\end{align*}
Using (\ref{2.16}), (\ref{3.10}) and $\psi\in\mathcal{F},$ we obtain
\begin{equation*}
\|L_{\varepsilon}(\theta(\partial_{\psi}N_{\varepsilon})
(\psi))\|_*\leq C\e\|\theta\|_*.
\end{equation*}
Consequently, $\partial_{\psi}B(\Lambda,\bar{Q},\phi)$ is invertible with uniformly bounded inverse. Then the fact that $(\Lambda,\bar{Q})\mapsto\phi(\Lambda,\bar{Q})$ is $C^1$ follows from the fact that $(\Lambda,\bar{Q},\psi)\mapsto L_{\e}(N_{\varepsilon}
(\psi))$ is $C^1$ and the implicit function theorem.

Finally, let us consider (\ref{4.14}). Differentiating (\ref{4.17}) with respect to $\Lambda,$ we find
$$\partial_{\Lambda}\phi=(\partial_{\psi}B(\Lambda,\xi,\phi))^{-1}\Big((\partial_{\Lambda}L_{\varepsilon})
(N_{\varepsilon}(\phi))+L_{\varepsilon}((\partial_{\Lambda}N_{\varepsilon})(\phi))+L_{\varepsilon}
(\partial_{\Lambda}R^{\varepsilon})\Big).$$
Then by Proposition \ref{pr3.1},
$$\|\partial_{\Lambda}\phi\|_*\leq C(\|N_{\varepsilon}(\phi)\|_{**}+\|(\partial_{\Lambda}N_{\varepsilon})(\phi)\|_{**}+\|\partial_{\Lambda}
R^{\varepsilon}\|_{**}).$$
From Lemma \ref{4.1} and (\ref{4.13}), we know that $\|N_{\varepsilon}(\phi)\|_{**}\leq C\varepsilon^2.$ Concerning the next term, we notice that according to the definition of $N_{\varepsilon},$
$$|\partial_{\Lambda}N_{\varepsilon}(\phi)|=3\phi^2|\pt_{\Lambda} W|.$$
Note that
$$|\partial_{\Lambda}W(z)|\leq C(\l z-\bar{Q}\r^{-2}+\e^2(-\ln \e)^{\frac 12}),$$
we have
$$\|\partial_{\Lambda}N_{\varepsilon}(\phi)\|_{**}\leq C\varepsilon.$$
Finally, using (\ref{4.5}), we obtain
$$\|\partial_{\Lambda}\phi\|_*\leq C\e.$$
The derivative of $\phi$ with respect to $\bar{Q}$ may be estimated in the same way.  This concludes the proof.
\end{proof}
\vspace{1cm}

\section{Finite-dimensional reduction: reduced energy}
Let us define a reduced energy functional as
\begin{align}
\label{5.1}
I_{\varepsilon}(\Lambda,Q)\equiv J_{\varepsilon}[W_{\Lambda,\bar{Q}}+\phi_{\varepsilon,\Lambda,\bar{Q}}]
\end{align}
for $n=4$ and
\begin{align}
\label{5.2}
I_{\varepsilon}(\Lambda,\eta,Q)\equiv J_{\varepsilon}[W_{\Lambda,\eta,\bar{Q}}+\phi_{\varepsilon,\Lambda,\eta,\bar{Q}}]
\end{align}
for $n=6.$ We have
\begin{proposition}
\label{pr5.1}
The function $u=W_{\Lambda,\bar{Q}}+\phi_{\e,\Lambda,\bar{Q}}$ is a solution to problem (\ref{1.11}) for $n=4$ if and only if $(\Lambda,\bar{Q})$ is a critical point of $I_{\varepsilon}.$ The function $u=W_{\Lambda,\eta,\bar{Q}}+\phi_{\e,\Lambda,\eta,\bar{Q}}$ is a solution to problem (\ref{1.11}) for $n=6$ if and only if $(\Lambda,\eta,\bar{Q})$ is a critical point of $I_{\varepsilon}.$
\end{proposition}

\begin{proof}
Here we only give the proof for the case $n=6,$ the other case can be proved in the same way. We notice that $u=W+\phi$ being a solution of (\ref{1.11}) is equivalent to being a critical point of $J_{\varepsilon}$, which is also equivalent to the vanish of the $d_i$'s in (\ref{4.7}) or, in view of
\begin{align}
\label{5.3}
&\langle Z_0,Y_0\rangle=\|Y_0\|_{\varepsilon}^2=\gamma_0+o(1),\nonumber\\
&\langle Z_i,Y_i\rangle=\|Y_i\|_{\varepsilon}^2=\gamma_1+o(1),~1\leq i\leq 6,\nonumber\\
&\langle Z_7,Y_7\rangle=\|Y_7\|_{\varepsilon}^2=\gamma_2\e^3,
\end{align}
where $\gamma_0,\gamma_1,\gamma_2$ are strictly positive constants, and
\begin{equation}
\label{5.4}
\langle Z_i,Y_j\rangle=o(1),i\neq j,0\leq i,j\leq6,\quad \l Z_i,Y_j\r=O(\e^3),i\neq j,i=7~\mathrm{or}~j=7.
\end{equation}
We have
\begin{align}
\label{5.5}
J_{\varepsilon}'[W+\phi][Y_i]=0,\quad0\leq i\leq 7.
\end{align}
On the other hand, we deduce from (\ref{5.2}) that $I'_{\varepsilon}(\Lambda,\eta,Q)=0$ is equivalent to the cancellation of $J_{\varepsilon}'(W+\phi)$ applied to the derivative of $W+\phi$ with respect to $\Lambda$, $\eta$ and $\bar{Q}.$ By the definition of $Y_i$'s and Proposition \ref{pr4.1}, we have
$$\frac{\partial(W+\phi)}{\partial\Lambda}=Y_0+y_0,\quad\frac{\partial(W+\phi)}{\partial\bar{Q}_i}=Y_i+y_i,\quad1\leq i\leq 6,\quad\frac{\partial (W+\phi)}{\partial\eta}=Y_7+y_7$$
with $\|y_i\|_{***}=O(\e^2),~0\leq i\leq 7$. We write
$$-\Delta(W+\phi)+\mu\e^2(W+\phi)-24(W+\phi)^2=\sum_{j=0}^7\alpha_jZ_j$$
and denote $a_{ij}=\l y_i,Z_j\r.$ Since $J_{\varepsilon}'[W+\phi][\theta]=0$ for $\langle\theta,Z_i\rangle=(\theta,Y_i)_{\varepsilon}=0,~0\leq i\leq 7,$ it turns out that $I_{\varepsilon}'(\Lambda,\eta,\bar{Q})=0$ is equivalent to
$$([b_{ij}]+[a_{ij}])[\alpha_j]=0,$$
where $b_{ij}=\l Y_i,Z_j\r.$ Using the estimate $\|y_i\|_{***}=O(\e^2)$ and the expression of $Z_i,Y_i,0\leq i\leq7,$ we directly obtain
\begin{align*}
&b_{00}=\gamma_0+o(1),\quad b_{ii}=\gamma_1+o(1)~\mathrm{for}~1\leq i\leq6,\quad b_{77}=\gamma_2\e^3,\\
&b_{ij}=o(1)~\mathrm{for}~0\leq i\neq j\leq6,\quad b_{ij}=O(\e^3)~\mathrm{for}~i=7~\mathrm{or}~j=7,i\neq j,\\
&a_{ij}=O(\e^2)~\mathrm{for}~0\leq i\leq7,0\leq j\leq 6,\quad a_{i7}=O(\e^4)~\mathrm{for}~0\leq i\leq7.
\end{align*}
Then it is easy to see the matrix $[b_{ij}+a_{ij}]$ is invertible by the above estimates of each components, hence $\alpha_i=0.$ We see that $I'_{\varepsilon}(\Lambda,\eta,\bar{Q})=0$ means exactly that (\ref{5.5}) is satisfied.
\end{proof}

With Proposition \ref{pr5.1}, it remains to find critical points of $I_{\varepsilon}.$ First, we establish an expansion of $I_{\varepsilon}.$
\begin{proposition}
\label{pr5.2}
In the case $n=4$, for $\varepsilon$ sufficiently small, we have
\begin{align}
\label{5.6}
I_{\e}(\Lambda,\eta,Q)=J_{\e}[W]+\e^2\big(\frac{c_1}{-\ln\e}\big)^{\frac12}\sigma_{\e,4}(\Lambda,Q)
\end{align}
where $\sigma_{\varepsilon,4}=O(\Lambda^2)+o(1)$ and $D_{\Lambda}(\sigma_{\varepsilon,4})=O(\Lambda)+o(1)$ as $\e$ goes to $0$, uniformly with respect to $\Lambda,~Q$ satisfying (\ref{2.11}).

In the case $n=6$, for $\e$ sufficiently small, we have
\begin{align}
\label{5.7}
I_{\e}(\Lambda,\eta,Q)=J_{\e}[W]+\e^4\sigma_{\e,6}(\Lambda,\eta,Q)
\end{align}
where $\sigma_{\varepsilon,6}=o(1)$ and $D_{\Lambda,\eta}(\sigma_{\varepsilon,6})=o(1)$ as $\e$ goes to $0$, uniformly with respect to $\Lambda,~\eta,~Q$ satisfying (\ref{2.13}).
\end{proposition}

\begin{proof}
We only consider the case $n=4$ here, the case $n=6$ can be argued similarly with minor changes. In view of (\ref{5.1}), a Taylor expansion and the fact that $J_{\varepsilon}'[W+\phi][\phi]=0$ yield
\begin{align*}
I_{\varepsilon}(\Lambda,Q)-J_{\varepsilon}[W]=&J_{\varepsilon}[W+\phi]-J_{\varepsilon}[W]=
-\int_0^1J_{\varepsilon}''(W+t\phi)[\phi,\phi](t)\mathrm{d}t\\
=&-\int_0^1(\int_{\Omega_{\varepsilon}}(|\nabla\phi|^2+\mu\e^2\phi^2-24(W+t\phi)^2\phi^2))t\mathrm{d}t,
\end{align*}
whence
\begin{align}
\label{5.9}
&I_{\varepsilon}(\Lambda,Q)-J_{\varepsilon}[W]\nonumber\\
&=-\int_0^1\Big{(}8\int_{\Omega_{\varepsilon}}(N_{\varepsilon}(\phi)\phi+3[W^2-(W+t\phi)^2]\phi^2)\Big{)}t\mathrm{d}t
-\int_{\Omega_{\varepsilon}}R^{\varepsilon}\phi.
\end{align}

The first term on the right hand side of (\ref{5.9}) can be estimated as
$$\Big|\int_{\Omega_{\varepsilon}}N_{\varepsilon}(\phi)\phi\Big|\leq C\int_{\Omega_{\varepsilon}}|\phi|^4+|W\phi^3|=O(\e^4\ln\e).$$
Similarly, for the second term on the right hand side of (\ref{5.9}), we obtain
$$\Big|\int_{\Omega_{\varepsilon}}[W^2-(W+t\phi)^2]\phi^2\Big|\leq C\int_{\Omega_{\varepsilon}}|\phi|^4+|W\phi^3|=O(\e^4\ln\e).$$
Concerning the last one, recalling
\begin{align*}
|R^{\e}|_*=|S_{\e}[W]|=&O\Big(\e^4(-\ln\e)\l z-\bar{Q}\r^{-2}+\e^2(-\ln\e)^{\frac12}\l z-\bar{Q}\r^{-4}\Big)\\
&+O(\Lambda)\Big(\frac{\e^4}{(-\ln\e)}|\ln\frac{1}{\e(1+|z-\bar{Q}|)}|+\frac{\e^4}{(-\ln\e)}\Big)
\end{align*}
uniformly in $\Omega_{\varepsilon}.$ A  simple computation shows that
$$\Big|\int_{\Omega_{\varepsilon}}R^{\varepsilon}\phi\Big|
=O\Big(\e^2(-\ln\e)^{\frac12}\Lambda^2+\e^3(-\ln\e)^{\frac12}\Big),$$
where we used $\|\phi\|_*=O(\e\Lambda).$ This concludes the proof of the first part of Proposition (\ref{5.6}).

An estimate for the derivatives with respect to $\Lambda$ is established exactly in the same way, differentiating the right side in (\ref{5.9}) and
estimating each term separately, using (\ref{4.3}), (\ref{4.5}) and Lemma \ref{le2.1}.
\end{proof}
\vspace{1cm}

\section{Proof of Theorem \ref{th1.6}}
In this section, we prove the existence of a critical point of $I_{\e}(\Lambda,Q)$ and $I_{\e}(\Lambda,\eta,Q)$, and then prove Theorem \ref{th1.6} by Proposition \ref{pr5.1}. According to Proposition \ref{pr5.2} and Lemma \ref{le2.1}. Setting
\begin{align}
\label{6.1}
K_{\e}(\Lambda,Q)=\frac{I_{\e}(\Lambda,Q)-2\int_{\B}U^4}{(-\frac{\ln\e}{c_1})^{\frac12}\e^2}
\end{align}
and
\begin{align}
\label{6.2}
K_{\e}(\Lambda,\eta,Q)=\frac{I_{\e}(\Lambda,\eta,Q)-4\int_{\B}U^3}{\e^3}
\end{align}
Then, we have when $n=4,$
\begin{align}
\label{6.3}
K_{\e}(\Lambda,Q)=&\frac {1}{4}c_4\Lambda^2\ln \frac {1}{\Lambda \e}(\frac {c_1}{-\ln \e})-\frac {c_4^2\Lambda^2}{2|\Omega|}+\frac 12c_4^2\Lambda^2H(Q,Q)(\frac {c_1}{-\ln \e})^{\frac{1}{2}}\nonumber\\
&+O\big(\frac{\Lambda^2}{-\ln\e}+\e\big),
\end{align}
and when $n=6,$
\begin{align}
\label{6.4}
K_{\e}(\Lambda,\eta,Q)=&\Big(\frac12\eta^2|\Omega|-c_6\Lambda^2\eta+\frac{1}{48}c_6\Lambda^2-8\eta^3|\Omega|\Big)
+\frac12c_6^2\Lambda^4H(Q,Q)\e\nonumber\\&
+\frac12\big(\eta-\frac{c_6\Lambda^2}{|\Omega|}\big)\e\int_{\Omega}\frac{\Lambda^2}{|x-Q|^4}+o(\e).
\end{align}

Then we begin to consider $K_{\e}(\Lambda,Q)$, finding its critical points with respect to $\Lambda, Q,$ and $K_{\e}(\Lambda,\eta,Q)$ with its critical points with respect to the parameters $\Lambda,\eta,Q$.

First, we consider $K_{\e}(\Lambda,Q)$ for $n=4$. For the setting of the parameters $\Lambda,Q,$ we see that $\Lambda,Q$ are located on a compact set, we can obtain a maximal value of $K_{\e}(\Lambda,Q).$  We claim that:
\medskip

\noindent {\bf Claim:} The maximal point of $K_{\e}(\Lambda,Q)$ with respect to $\Lambda,Q$ can not happen on the boundary of the parameters.
\medskip

If we can prove this claim, then we could obtain an interior critical point of $K_{\e}(\Lambda,Q).$ Before proving the claim, we first consider
$$F_{\e}(\Lambda)=\frac {1}{4}c_4\Lambda^2\ln \frac {1}{\Lambda \e}(\frac {c_1}{-\ln \e})-\frac {c_4^2\Lambda^2}{2|\Omega|}.$$
Note that
$$\frac {\pt}{\pt \Lambda}[F_{\e}(\Lambda)]=\frac {1}{2}c_4\Lambda\ln \frac {1}{\Lambda \e}(\frac {c_1}{-\ln \e})-\frac{1}{4}c_4\Lambda(\frac {c_1}{-\ln \e})-\frac {c_4^2\Lambda}{|\Omega|},$$
Choosing $c_1=\frac {2c_4}{|\Omega|}$, we could obtain that there exists
$$\Lambda^*=\exp(-\frac12)\in (\exp(-\frac12)\e^{\beta},\exp(-\frac12)\e^{-\beta})$$
with some proper fixed constant $\beta\in(0,\frac13)$, such that
$$\frac {\pt}{\pt \Lambda}F_{\e}\mid_{\Lambda=\Lambda^*}=0.$$
It can be also found that such $\Lambda^*$ provides the maximal value of $F_{\e}(\Lambda)$ in $[\Lambda_{4,1},\Lambda_{4,2}],$ where $\Lambda_{4,1}=\exp(-\frac12)\e^{\beta},\Lambda_{4,2}=\exp(-\frac12)\e^{-\beta}.$ In order to prove the claim, we need to take $\Lambda$ into consideration for the expansion of the energy, going through the first part of the Appendix, we have
\begin{align*}
K_{\e}(\Lambda,Q)=~&\frac {1}{4}c_4\Lambda^2\ln \frac {1}{\Lambda \e}(\frac {c_1}{-\ln \e})-\frac {c_4^2\Lambda^2}{2|\Omega|}+\frac 12c_4^2\Lambda^2H(Q,Q)(\frac {c_1}{-\ln \e})^{\frac{1}{2}}\nonumber\\
&+O(\frac{\Lambda^2}{-\ln\e}+\e).
\end{align*}

Now, we come back to prove the claim, choosing $\Lambda=\Lambda^*$ and $Q=p.$ (Here $p$ refers to the point where $H(Q,Q)$ obtain its maximal value, it is possible to find such a point. Indeed, we notice a fact $H(Q,Q)\rightarrow-\infty$ as $d(Q,\partial\Omega)\rightarrow0$ see \cite{rw} and references therein for a proof of this fact. Therefore we could find such $p$.)

First, we prove that the maximal value can not happen on $\partial \mathcal{M}_{\delta_4}.$  We choose $\delta_4$ such that $\omega_1<\max_{\partial\mathcal{M}_{\delta_4}}H<\omega_2$ for some proper constant $\omega_1,\omega_2$ sufficiently negative, then we fixed $\mathcal{M}_{\delta_4}.$ It is easy to see that $K_{\e}(\Lambda,Q)<K_{\e}(\Lambda,p),$ where $Q$ lies on the boundary of $\mathcal{M}_{\delta_4}$ and $\Lambda\in(\Lambda_{4,1},\Lambda_{4,2}).$ For $\Lambda=\Lambda_{4,1}$ or $\Lambda_{4,2},$ we go to the arguments below. Therefore, we prove that the maximal point can not lie on the boundary of $\mathcal{M}_{\delta_4}\times[\Lambda_{4,1},\Lambda_{4,2}].$

Next, we show $K_{\e}(\Lambda^*,p)>K_{\e}(\Lambda_{4,2},Q).$ It is easy to see that
$$F_{\e}[\Lambda_{4,2}]\leq c\e^{-2\beta},$$
where $c<0.$ Then we can find $c_1<0$ such that $K_{\e}(\Lambda_{4,2},Q)\leq c_1\e^{-2\beta}$ for any $Q\in\mathcal{M}_{\delta_4}$, since the other terms compared to $\e^{-2\beta}$ are higher order term. On the other hand, for the choice of $\Lambda^*,p,$ we see that $K_{\e}(\Lambda^*,p)=O(1).$ Therefore, we prove that $K_{\e}(\Lambda^*,p)>K_{\e}(\Lambda_{4,2},Q)$ for any $Q\in\mathcal{M}_{\delta_4}$.

It remains to prove that the maximal value can not happen at $\Lambda=\Lambda_{4,1}$. We choose $\Lambda=\e^{\beta/2},Q=p,$ direct computation yields.
\begin{align*}
K_{\e}(\e^{\beta/2},p)=\frac{\beta c_4^2\e^{\beta}}{4|\Omega|}(1+o(1)),\quad
K_{\e}(\Lambda_{4,1},Q)=\frac{\beta c_4^2\e^{2\beta}}{2|\Omega|}(1+o(1)).
\end{align*}
It is to see $K_{\e}(\e^{\beta/2},p)>K_{\e}(\Lambda_{4,1},Q)$ for any $Q\in\mathcal{M}_{\delta_4}$ when $\e$ is sufficiently small. Hence, we finish the proof of the claim. In other words, we could obtain an interior maximal point in $[\Lambda_{4,1},\Lambda_{4,2}]\times\mathcal{M}_{\delta_4}.$ Therefore, we show the existence of the critical points of $K_{\e}(\Lambda,Q)$ with respect to $\Lambda,Q.$

\medskip

For $n=6$.  We set $\eta=\frac{1}{48}+a\e^{\frac13},$ $\frac{c_6\Lambda^2}{|\Omega|}=\frac{1}{96}+b\e^{\frac23},$ then
\begin{align}
\label{6.5}
K_{\e}(a,b,Q):=K_{\e}(\Lambda,\eta,Q)=\frac{1}{6912}|\Omega|+\big{[}F(Q)-(8a^3+ab)|\Omega|\big{]}\e+o(\e),
\end{align}
where
\begin{align*}
F(x)=\frac{|\Omega|}{18432}\Big(|\Omega|H(x,x)+\frac{1}{c_6}\int_{\Omega}\frac{1}{|x-y|^4}\mathrm{d}y\Big),
\end{align*}
$-\eta_6\leq a\leq\eta_6$ and $-\Lambda_6\leq b\leq\Lambda_6.$

We set $C_0=F(p_0),$ $p_0$ refers to the point where $F(x)$ obtains its maximal value. Indeed, we have $H(Q,Q)\rightarrow-\infty$ as $d(Q,\partial\Omega)\rightarrow0$ and $I(x)=\int_{\Omega}\frac{1}{|x-y|^4}\mathrm{d}y$ is uniformly bounded in $\Omega.$ Hence, we can always find such point $p_0.$ Let us introduce another five constants $C_i,i=1,2,3,4,5,$ with $C_2<C_1<C_0$, $0<C_3<C_4<\eta_6$ and $0<C_3<C_5<\Lambda_6,$ the value of these five constants will be determined later.

We set
\begin{equation}
\label{6.6}
\Sigma_0=\Big{\{}-C_4\leq a\leq C_4,~-C_5\leq b\leq C_5,~Q\in \mathcal{N}_{C_2}\Big{\}},
\end{equation}
where $\mathcal{N}_{C_i}=\{q:F(q)>C_i\},i=1,2$ and $\delta_6$ is chosen such that $\mathcal{N}_{C_2}\subset\mathcal{M}_{\delta_6}.$

We also define
\begin{equation}
\label{6.7}
B=\{(a,b,Q)\mid(a,b)\in B_{C_3}(0),~Q\in\overline{\mathcal{N}_{C_1}}\},~B_0=\{(a,b)\mid (a,b)\in B_{C_3}(0)\}\times\partial\mathcal{N}_{C_1},
\end{equation}
where $B_{r}(0):=\{0\leq a^2+b^2\leq r\}.$

It is trivial to see that $B_0\subset B\subset\Sigma_0$, $B$ is compact. Let $\Gamma$ be the class of continuous functions $\varphi:B\rightarrow\Sigma_0$ with the property that $\varphi(y)=y,y=(a,b,Q)$ for all $y\in B_0.$ Define the min-max value $c$ as
$$c=\min_{\varphi\in\Gamma}\max_{y\in B}K_{\e}(\varphi(y)).$$
We now show that $c$ defines a critical value. To this end, we just have to verify the following conditions

\begin{itemize}
  \item [(T1)] $\max_{y\in B_0}K_{\e}(\varphi(y))<c,~\forall\varphi\in\Gamma,$
  \item [(T2)] For all $y\in\pt \Sigma_0$ such that $K_{\e}(y)=c,$ there exists a vector $\tau_y$ tangent to $\partial\Sigma_0$ at $y$ such that
  $$\partial_{\tau_y}K_{\e}(y)\neq0.$$
\end{itemize}

Suppose (T1) and (T2) hold. Then standard deformation argument ensures that the min-max value $c$ is a (topologically nontrivial) critical value for $K_{\e}(\Lambda,\eta,Q)$ in $\Sigma_0.$ (Similar notion has been introduced in \cite{dfw} for degenerate critical points of mean curvature.)

To check (T1) and (T2), we define $\varphi(y)=\varphi(a,b,Q)=(\varphi_a,\varphi_b,\varphi_Q)$ where $(\varphi_a,\varphi_b)\in [-C_4,C_4]\times[-C_5,C_5]$ and $\varphi_Q\in\mathcal{N}_{C_2}.$

For any $\varphi\in\Gamma$ and $Q\in\mathcal{N}_{C_2},$ the map $Q\rightarrow\varphi_Q(a,b,Q)$ is a continuous function from $\mathcal{N}_{C_1}$ to $\mathcal{N}_{C_2}$ such that $\varphi_Q(a,b,Q)=Q$ for $Q\in\partial\mathcal{N}_{C_1}.$ Let $\mathcal{D}$ be the smallest ball which contain $\mathcal{N}_{C_1}$, we extend $\varphi_Q$ to a continuous function $\tilde{\varphi}_Q$ from $\mathcal{D}$ to $\mathcal{D}$ where $\tilde{\varphi}(Q)$ is defined as follows:
$$\tilde{\varphi}_Q(x)=\varphi(x),~x\in\mathcal{N}_{C_1},\quad\tilde{\varphi}_Q(x)=Id,~x\in\mathcal{D} \setminus\mathcal{N}_{C_1}.$$
Then we claim there exists $Q'\in\mathcal{D}$ such that $\tilde{\varphi}_{Q}(Q')=p_0.$ Otherwise $\frac{\tilde{\varphi}_{Q}-p_0}{|\tilde{\varphi}_{Q}-p_0|}$ provides a continuous map from $\mathcal{D}$ to $\mathbb{S}^5,$ which is impossible in algebraic topology. Hence, there exists $Q'\in\mathcal{D}$ such that $\tilde{\varphi}_{Q}(Q')=p_0.$ By the definition of $\tilde{\varphi},$ we can further conclude $Q'\in\mathcal{N}_{C_1}.$ Whence
\begin{align}
\label{6.8}
\max_{y\in B}K_{\e}(\varphi(y))\geq&~ K_{\e}(\varphi_a(a,b,Q'),\varphi_b(a,b,Q'),p_0)\nonumber\\
\geq&~\frac{1}{6912}|\Omega|+(C_0-C_6|\Omega|)\e+o(\e),
\end{align}
where $C_6=8C_4^3+C_4C_5$ which stands for the maximal value of $8a^3+ab$ in $[-C_4,C_4]\times[-C_5,C_5].$ As a consequence
\begin{equation}
\label{6.9}
c\geq\frac{1}{6912}|\Omega|+(C_0-C_6|\Omega|)\e+o(\e).
\end{equation}

For $(a,b,Q)\in B_0,$ we have $F(\varphi_{Q}(a,b,Q))=C_1.$ So we have
\begin{equation}
\label{6.10}
K_{\e}(a,b,Q)\leq\frac{1}{6912}|\Omega|+(C_1+C_7|\Omega|)\e+o(\e),
\end{equation}
where $C_7=\max_{(a,b)\in B_{C_3}(0)}8a^3+ab<8C_3^3+C_3^2.$

If we choose $C_0-C_1>8C_4^3+C_4C_5+8C_3^3+C_3^2>C_6+C_7.$ Then $\max_{y\in B_0}K_{\e}(\varphi(y))<c$ holds. So (T1) is verified.

To verify (T2), we observe that
$$\partial\Sigma_0=:\{a,b,Q\mid a=-C_4~\mathrm{or}~a=C_4~\mathrm{or}~b=-C_5~\mathrm{or}~b=C_5~\mathrm{or}~Q\in\partial\mathcal{N}_{C_2}\}.$$
Since $C_4,C_5$ are arbitrary, we choose $0<24C_4^2<C_5.$ Then on $a=-C_4$ or $a=C_4$, we choose $\tau_y=\frac{\pt}{\pt b}$, on $b=-C_5$ or $b=C_5$, we choose $\tau_y=\frac{\pt}{\pt a}.$ By our setting on $C_4,C_5$, we could show $\partial_{\tau_y}K_{\e}(y)\neq0.$ It only remains to consider the case $Q\in\partial\mathcal{N}_{C_2}.$ If $Q\in\partial\mathcal{N}_{C_2},$ then
\begin{align}
\label{6.11}
K_{\e}(a,b,Q)\leq\frac{1}{6912}|\Omega|+ (C_2+C_7|\Omega|)\e+o(\e),
\end{align}
which is obviously less than $c$ for $C_2<C_1.$ So (T2) is also verified.

In conclusion, we proved that for $\e$ sufficiently small, $c$ is a critical value, i.e., a critical point $(a,b,Q)\in\Sigma_0$ of $K_{\e}$ exists. Which means $K_{\e}$ indeed has critical points respect to $\Lambda,\eta,Q$ in (\ref{2.13}).
\vspace{0.5cm}

\noindent{\it Proof of Theorem \ref{th1.6}.} For $n=4,$ we proved that for $\e$ small enough, $I_{\e}$ has a critical point $(\Lambda^{\e},Q^{\e}).$ Let $u_{\e}=W_{\Lambda^{\e},\bar{Q}^{\e},\e}.$ Then $u_{\e}$ is a nontrivial solution to problem (\ref{1.12}) for $n=4$. The strong maximal principle shows $u_{\e}>0$ in $\O.$ Let $u_{\mu}=\e^{-1}u_{\e}(x/\e)$. By our construction, $u_{\mu}$ has all the properties stated in Theorem \ref{th1.6}.

For $n=6,$ we proved that for $\e$ small enough, $I_{\e}$ has a critical point $(\Lambda^{\e},\eta^{\e},Q^{\e}).$ Let $u_{\e}=W_{\Lambda^{\e},\eta^{\e},\bar{Q}^{\e},\e}.$ Then $u_{\e}$ is a nontrivial solution to problem (\ref{1.12}) for $n=6$. The strong maximal principle shows $u_{\e}>0$ in $\O.$ Let $u_{\mu}=\e^{-2}u_{\e}(x/\e)$. By our construction, $u_{\mu}$ has all the properties stated in Theorem \ref{th1.6}. $\square$
\vspace{1cm}

\section{Appendix A:~Proof of Lemma ~2.1}
We divide the proof into two parts. First, we study the case $n=4.$ From the definition of $W,$ (\ref{2.10}) and (\ref{2.15}), we know that
\begin{align*}
S_{\e}[W]=&-\Delta W+\mu\e^2W-8W^3\\
=&~8U^3+\e^4\big(\frac{c_1}{-\ln\e}\big)\hat{U}
-\e^2\big(\frac{c_1}{-\ln\e}\big)^{\frac12}\Delta(R_{\e,\Lambda,Q}\chi)-8W^3\\
=&~O\Big{(}\e^4(-\ln\e)\l z-\bar{Q}\r^{-2}+\e^2(-\ln\e)^{\frac12}\l z-\bar{Q}\r^{-4}\Big)\\
&~+O(\Lambda)\Big(\frac{\e^4}{(-\ln\e)}\big|\ln\frac{1}{\e(1+|z-\bar{Q}|)}\big|+\frac{\e^4}{(-\ln\e)^{\frac12}}\Big{)}.
\end{align*}
The estimates for $D_{\Lambda}S_{\e}[W]$ and $D_{\bar{Q}}S_{\e}[W]$ can be computed in the same way.

We now turn to the proof of the energy estimate (\ref{2.23}). From (\ref{2.15}) and (\ref{2.16}) we deduce that
\begin{align}
\label{7.1}
\int_{\O}|\nabla W|^2+\e^2\big(\frac{c_1}{-\ln\e}\big)^{\frac12}\int_{\O}W^2=
&8\int_{\O}U^3W+\e^4\big(\frac{c_1}{-\ln\e}\big)\int_{\O}\hat{U}W\nonumber\\
&-\e^2\big(\frac{c_1}{-\ln\e}\big)^{\frac12}\int_{\O}\Delta(R\chi)W.
\end{align}

Concerning the first term on the right hand side of (\ref{7.1}), we have
\begin{align}
\label{7.2}
\int_{\O}U^3W=\int_{\O}U^4+\e^2\big(\frac{c_1}{-\ln\e}\big)^{\frac12}\int_{\O}\hat{U}U^3
+\frac{c_4\Lambda}{|\Omega|}\e^2\big(\frac{c_1}{-\ln\e}\big)^{-\frac12}\int_{\O}U^3.
\end{align}
We note that
\begin{align*}
\int_{\O}U^4=\int_{\mathbb{R}^4}U_{1,0}^4+O(\e^4),\quad \int_{\O}U^3=\frac{c_4\Lambda}{8}+O(\e^2).
\end{align*}
Then, we get
\begin{align*}
\int_{\O}U^3W=&\int_{\mathbb{R}^4}U_{1,0}^4+\frac{c_4^2\Lambda^2}{8|\Omega|}\e^2\big(\frac{c_1}{-\ln\e}\big)^{-\frac12}
+\e^2\big(\frac{c_1}{-\ln\e}\big)^{\frac12}\int_{\O}\hat{U}U^3+O\big(\e^4\big(\frac{c_1}{-\ln\e}\big)^{-\frac12}\big),
\end{align*}
for the third term on the right hand side of the above equality, we have
\begin{align*}
\int_{\O}\hat{U}U^3=&-\int_{\O}\Psi U^3-c_4\Lambda\big(\frac{c_1}{-\ln\e}\big)^{-\frac12}\int_{\O}H(x,Q)U^3+\int_{\O}(R\chi)U^3\\
=&-\frac{c_4\Lambda^2}{16}\ln\frac{1}{\Lambda\e}-\frac{c_4^2\Lambda^2}{8}\big(\frac{c_1}{-\ln\e}\big)^{-\frac12}H(Q,Q)+O(\Lambda^2).
\end{align*}
Hence, we have
\begin{align}
\label{7.3}
\int_{\O}U^3W=&\int_{\mathbb{R}^4}U_{1,0}^4+\frac{c_4^2\Lambda^2}{8|\Omega|}\e^2\big(\frac{c_1}{-\ln\e}\big)^{-\frac12}-
\frac{c_4\Lambda^2}{16}\ln\frac{1}{\Lambda\e}\e^2\big(\frac{c_1}{-\ln\e}\big)^{\frac12}\nonumber\\
&-\frac{c_4^2\Lambda^2}{8}\e^2H(Q,Q)+O\big(\e^2\big(\frac{c_1}{-\ln\e}\big)^{\frac12}\Lambda^2
+\e^4\big(\frac{c_1}{-\ln\e}\big)^{-\frac12}\big).
\end{align}
For the second term on the right hand side of (\ref{7.1})
\begin{equation*}
\int_{\O}\hat{U}W=\int_{\O}\hat{U}U+\e^2\big(\frac{c_1}{-\ln\e}\big)^{\frac12}\int_{\O}\hat{U}^2
+\frac{c_4\Lambda}{|\Omega|}\e^2\big(\frac{c_1}{-\ln\e}\big)^{-\frac12}\int_{\O}\hat{U}.
\end{equation*}
By noting that
\begin{align*}
&\int_{\O}\hat{U}U=O\big(\e^{-2}\big(\frac{c_1}{-\ln\e}\big)^{-\frac12}\Lambda^2\big),
~\int_{\O}\hat{U}^2=O\big(\e^{-4}(-\ln\e)\Lambda^2\big),\\
&\int_{\O}\hat{U}=\e^{-4}(\frac{c_1}{-\ln\e})^{-\frac12}\int_{\Omega}\frac{\Lambda}{|x-Q|^2}+O(\e^{-4}\Lambda),
\end{align*}
where we used $\int_{\Omega}G(x,Q)=0.$ Then, we obtain
\begin{align}
\label{7.4}
\e^4(\frac{c_1}{-\ln\e})\int_{\O}\hat{U}W=\frac{c_4\Lambda^2}{|\Omega|}\e^2
\int_{\Omega}\frac{1}{|x-Q|^2}+O\big(\e^2\big(\frac{c_1}{-\ln\e}\big)^{\frac12}\Lambda^2\big).
\end{align}
For the last term on the right hand side of (\ref{7.1}),
\begin{align}
\label{7.5}
\int_{\O}\Delta(R\chi)W=&\e^2(\frac{c_1}{-\ln\e})^{-\frac12}\frac{c_4\Lambda}{|\Omega|}
\int_{\O}\Delta(R\chi)+O(\Lambda^2)\nonumber\\
=&\e^2(\frac{c_1}{-\ln\e})^{-\frac12}\frac{c_4\Lambda}{|\Omega|}\int_{\pt\O}\frac{\pt(R\chi)}{\pt\nu}+O(\Lambda^2)
\nonumber\\
=&\big(\frac{c_1}{-\ln\e}\big)^{-1}\frac{c_4\Lambda}{|\Omega|}
\int_{\pt\O}\frac{\pt(U-\e^2(\frac{c_1}{-\ln\e})^{\frac12}\Psi-c_4\Lambda\e^2H)}{\pt\nu}+O(\Lambda^2)\nonumber\\
=&\big(\frac{c_1}{-\ln\e}\big)^{-1}\frac{c_4\Lambda}{|\Omega|}
\int_{\O}\Delta\big(U-\e^2\big(\frac{c_1}{-\ln\e}\big)^{\frac12}\Psi-c_4\Lambda\e^2H\big)+O(\Lambda^2)\nonumber\\
=&\big(\frac{c_1}{-\ln\e}\big)^{-1}\frac{c_4\Lambda}{|\Omega|}
\int_{\O}\big(-8U^3+\e^2\big(\frac{c_1}{-\ln\e}\big)^{\frac12}U+c_4\Lambda\e^4\frac{1}{|\Omega|}\big)
+O(\Lambda^2)\nonumber\\
=&\big(\frac{c_1}{-\ln\e}\big)^{-\frac12}\frac{c_4\Lambda^2}{|\Omega|}\int_{\Omega}\frac{1}{(\e^2\Lambda^2+|x-Q|^2)}
+O(\Lambda^2+\e^2(-\ln\e)).
\end{align}
(\ref{7.3})-(\ref{7.5}) implies
\begin{align}
\label{7.6}
\frac{1}{2}\int_{\O}\Big{(}|\nabla W|^2+\e^2(\frac{c_1}{-\ln\e})^{\frac12}W^2\Big{)}=&4\int_{\mathbb{R}^4}U_{1,0}^4
+\e^2(\frac{c_1}{-\ln\e})^{-\frac12}\frac{c_4^2\Lambda^2}{2|\Omega|}-\frac{c_4^2\Lambda^2}{2}H(Q,Q)\e^2\nonumber\\
&-\frac{c_4\Lambda^2}{4}\e^2(\frac{c_1}{-\ln\e})^{\frac12}\ln\frac{1}{\Lambda\e}
+O(\e^2(\frac{c_1}{-\ln\e})^{\frac12}\Lambda^2)\nonumber\\
&+O\big(\e^4\big(\frac{c_1}{-\ln\e}\big)^{-\frac12}\big).
\end{align}

At last, we compute the term $\int_{\O}W^4.$
\begin{align}
\label{7.7}
\int_{\O}W^4=&\int_{\O}U^4+4\e^2\big(\frac{c_1}{-\ln\e}\big)^{\frac12}\int_{\O}U^3\hat{U}
+4\e^2\big(\frac{c_1}{-\ln\e}\big)^{-\frac12}\frac{c_4\Lambda}{|\Omega|}\int_{\O}U^3\nonumber\\
&+O\big(\e^4\big(\frac{c_1}{-\ln\e}\big)^{-2}\big)\nonumber\\
=&\int_{\mathbb{R}^4}U_{1,0}^4-\frac{c_4\Lambda^2}{4}\e^2\big(\frac{c_1}{-\ln\e}\big)^{\frac12}\ln\frac{1}{\Lambda\e}
-\frac{c_4^2\Lambda^2}{2}\e^2H(Q,Q)\nonumber\\
&+\frac{c_4^2\Lambda^2}{2|\Omega|}\e^2\big(\frac{c_1}{-\ln\e}\big)^{-\frac12}
+O\big(\e^2\big(\frac{c_1}{-\ln\e}\big)^{\frac12}\Lambda^2\big)\nonumber\\
&+O\big(\e^4\big(\frac{c_1}{-\ln\e}\big)^{-2}\big).
\end{align}

Combining (\ref{7.6}) and (\ref{7.7}), we obtain
\begin{align}
\label{7.8}
J_{\e}[W]=&\frac{1}{2}\int_{\O}|\nabla W|^2+\frac{\mu\e^2}{2}\int_{\O}W^2-2\int_{\O}W^4\nonumber\\
=&2\int_{\mathbb{R}^4}U_{1,0}^4+\frac{c_4\Lambda^2}{4}\e^2\big(\frac{c_1}{-\ln\e}\big)^{\frac12}\ln\frac{1}{\Lambda\e}
-\frac{c_4^2\Lambda^2}{2|\Omega|}\e^2
\big(\frac{c_1}{-\ln\e}\big)^{-\frac12}\nonumber\\
&+\frac{1}{2}c_4^2\Lambda^2\e^2H(Q,Q)+O\big(\e^2\big(\frac{c_1}{-\ln\e}\big)^{\frac12}\Lambda^2\big)\nonumber\\
&+O(\e^4(-\ln\e)^2).
\end{align}
\medskip

In the end of this section, we prove (\ref{2.24})-(\ref{2.28}). From the definition of $W$, (\ref{2.10}) and (\ref{2.15}), we know that
\begin{align*}
S_{\e}[W]=&-\Delta W+\e^3W-24W^2 \\
=&~24U^2+\e^6\hat{U}-\e^3\Delta(R\chi)+\e^6\big(\eta-\frac{c_6\Lambda^2}{|\Omega|}\big)-24U^2-24\eta^2\e^6+O\big(\e^3\l z-\bar{Q}\r^{-4}\big)\\
=&~-\e^6\big(24\eta^2-\eta+\frac{c_6\Lambda^2}{|\Omega|}\big)+O\big(\e^3\l z-\bar{Q}\r^{-4}\big)\\
=&~ O\big(\l z-\bar{Q}\r^{-3\frac23}\e^3\big).
\end{align*}
The estimates for $D_{\Lambda}S_{\e}[W],~D_{\bar{Q}}S_{\e}[W]$ and $D_{\eta}S_{\e}[W]$ can be derived in the same way. Now we are in the position to compute the energy. From (\ref{2.15}) and (\ref{2.16}), we deduce that
\begin{align}
\label{7.9}
\int_{\O}|\nabla W|^2+\e^3\int_{\O}W^2=&\int_{\O}(-\Delta W+\e^3W)W\nonumber\\
=&\int_{\O}\Big{(}24U^2+\e^6\hat{U}-\e^3\Delta(R\chi)+\e^6\big(\eta-\frac{c_6\Lambda^2}{|\Omega|}\big)\Big{)}W.
\end{align}

Concerning the first term on the right hand side of (\ref{7.9}), we have
\begin{align}
\label{7.10}
\int_{\O}U^2W=&\int_{\O}U^3+\e^3\int_{\O}\hat{U}U^2+\eta\e^3\int_{\O}U^2\nonumber\\
=&\int_{\mathbb{R}^6}U_{1,0}^3+\frac{1}{24}c_6\eta\Lambda^2\e^3
-\e^3\int_{\O}U^2\Psi-c_6\Lambda^2\e^4\int_{\O}U^2H\nonumber+O(\e^5)\\
=&\int_{\mathbb{R}^6}U_{1,0}^3+\frac{1}{24}c_6\eta\Lambda^2\e^3
-\frac{1}{24}c_6^2\Lambda^4\e^4H(Q,Q)-\frac{1}{576}c_6\Lambda^2\e^3+O(\e^5).
\end{align}

For the second, third and fourth term on the right hand side of (\ref{7.9}), following the similar steps as we did in case $n=4.$
\begin{align}
\label{7.11}
\e^6\int_{\O}\hat{U}W=\e^6\int_{\O}\hat{U}(U+\e^3\hat{U}+\eta\e^3)=
-\eta\Lambda^2\e^4\int_{\Omega}\frac{1}{|x-Q|^4}+O(\e^5),
\end{align}
\begin{align}
\label{7.12}
-\e^3\int_{\O}\Delta(R\chi)W=~&\e^3\eta\int_{\O}\Delta(U-\e^3\Psi-c_6\e^4\Lambda^2H)+O(\e^5)=
\e^6\eta\int_{\Omega_{\e}}U+O(\e^5)\nonumber\\
=~&\eta\Lambda^2\e^4\int_{\Omega}\frac{1}{|x-Q|^4}+O(\e^5),
\end{align}
and
\begin{align}
\label{7.13}
\e^6\big(\eta-\frac{c_6\Lambda^2}{|\Omega|}\big)\int_{\O}W=\big(\eta^2|\Omega|-c_6\eta\Lambda^2\big)\e^3+
\big(\eta-\frac{c_6\Lambda^2}{|\Omega|}\big)\e^4\int_{\Omega}\frac{\Lambda^2}{|x-Q|^4}+O(\e^5).
\end{align}

(\ref{7.10})-(\ref{7.13}) implies
\begin{align}
\label{7.14}
\frac12\int_{\O}|\nabla W|^2+\frac{\e^3}{2}\int_{\O}W^2=&12\int_{\mathbb{R}^6}U_{1,0}^3
+\big(\frac12\eta^2|\Omega|-\frac{1}{48}c_6\Lambda^2\big)\e^3-\frac{c_6^2\Lambda^4}{2}H(Q,Q)\e^4\nonumber\\
&+\frac12(\eta-\frac{c_6\Lambda^2}{|\Omega|})\e^4\int_{\Omega}\frac{\Lambda^2}{|x-Q|^4}+O(\e^5).
\end{align}

Then,
\begin{align}
\label{7.15}
\int_{\O}W^3=
&\int_{\mathbb{R}^6}U_{1,0}^3+3\e^3\int_{\O}U^2\hat{U}+3\e^3\int_{\O}U^2\eta
+3\e^6\int_{\O}U\eta^2+3\e^9\int_{\O}\hat{U}\eta^2\nonumber\\
&+\e^9\int_{\O}\eta^3+O(\e^5)\nonumber\\
=&\int_{\mathbb{R}^6}U_{1,0}^3+\frac18c_6\eta\Lambda^2\e^3-\frac{1}{192}c_6\Lambda^2\e^3+\eta^3|\Omega|\e^3
-\frac18c_6^2\Lambda^4H(Q,Q)\e^4\nonumber\\
&+O(\e^5).
\end{align}

Combining (\ref{7.14})-(\ref{7.15}), we gain the energy
\begin{align}
\label{7.16}
J_{\e}[W]=~&4\int_{\mathbb{R}^6}U_{1,0}^3+\Big{(}\frac12\eta^2|\Omega|-c_6\eta\Lambda^2+\frac{1}{48}c_6\Lambda^2
-8\eta^3|\Omega|\Big{)}\e^3+\frac12c_6^2\Lambda^4H(Q,Q)\e^4\nonumber\\
&+\frac12\big(\eta-\frac{c_6\Lambda^2}{|\Omega|}\big)\e^4
\int_{\Omega}\frac{\Lambda^2}{|x-Q|^4}+O(\e^5).
\end{align}

Hence, we finish the whole proof of Lemma \ref{le2.1}.         $\square$

\vspace{2cm}
\bigskip   \bibliographystyle{plain}

\end{document}